\documentclass[a4paper,10pt,oneside,fleqn]{amsart}
\usepackage{amsthm}
\usepackage{amsmath}
\usepackage{tikz}
\usepackage{bm}
\usepackage{bbm}
\usepackage{amsfonts}
\usepackage{nicematrix}
\usepackage{amssymb}
\usepackage{mathtools}
\usepackage{appendix}
\usepackage{mathrsfs}
\usepackage{setspace}
\usepackage{todonotes}
\usepackage{enumitem}
\usepackage[foot]{amsaddr}
\usepackage[pdfdisplaydoctitle,colorlinks,breaklinks,urlcolor=blue,linkcolor=blue,citecolor=blue]{hyperref}
\usepackage[nameinlink,capitalise]{cleveref}%must always be the last
\newcommand{\E}{\mathbb{E}}
\newcommand{\F}{\mathcal{F}}

\renewcommand{\H}{\mathbb{H}}

\newcommand{\N}{\mathbb{N}}

\renewcommand{\P}{\mathbb{P}}

\newcommand{\R}{\mathbb{R}}

\newcommand{\Z}{\mathbb{Z}}

\newcommand{\cA}{\mathcal{A}}

\newcommand{\cC}{\mathcal{C}}

\newcommand{\cF}{\mathcal{F}}
\newcommand{\cG}{\mathcal{G}}
\newcommand{\cH}{\mathcal{H}}

\newcommand{\cQ}{\mathcal{Q}}
\newcommand{\cR}{\mathcal{R}}
\newcommand{\cS}{\mathcal{S}}
\newcommand{\cT}{\mathcal{T}}
\newcommand{\cU}{\mathcal{U}}

\newcommand{\cX}{\mathcal{X}}

\DeclareMathOperator{\trace}{Tr}

\DeclareMathOperator{\e}{e}

\newcommand{\dd}{\,\mathrm{d}}
\newcommand{\eps}{\varepsilon}

\newtheorem{theorem}{Theorem}[section]
\newtheorem{definition}[theorem]{Definition}

\newtheorem{lemma}[theorem]{Lemma}
\newtheorem{proposition}[theorem]{Proposition}
\newtheorem{assumption}[theorem]{Assumption}

\theoremstyle{remark}
\newtheorem{remark}[theorem]{Remark}

\numberwithin{equation}{section}

\renewcommand{\leq}{\leqslant}
\renewcommand{\le}{\leqslant}
\renewcommand{\geq}{\geqslant}

\def\<{\langle}
\def\>{\rangle}
\def\eps{\varepsilon}

\title{Edwards-Wilkinson limit for a stochastic advection-diffusion PDE}
\author{Sotirios Kotitsas}
\thanks{Dipartimento di Matematica, Universit\`a di Pisa,
Largo Bruno Pontecorvo 5, 56127 Pisa, Italy.
\href{mailto:sotirios.kotitsas@dm.unipi.it}{sotirios.kotitsas@dm.unipi.it}}

\author{Dejun Luo}
\thanks{SKLMS, Academy of Mathematics and Systems Science, Chinese Academy of Sciences,
Beijing 100190, China; and School of Mathematical Sciences, University of Chinese Academy of Sciences,
Beijing 100049, China. \href{mailto:luodj@amss.ac.cn}{luodj@amss.ac.cn}}

\author{Mario Maurelli}
\thanks{Dipartimento di Matematica, Universit\`a di Pisa,
Largo Bruno Pontecorvo 5, 56127 Pisa, Italy.
\href{mailto:mario.maurelli@unipi.it}{mario.maurelli@unipi.it}} 

\date{}

\begin{document}

\begin{abstract}
We consider a diffusion in a Gaussian random environment that is white in time and study the large-scale behavior of the quenched density with respect to the Lebesgue measure. We show that under diffusive rescaling, the fluctuations of the density converge to a Gaussian limit, described by an additive stochastic heat equation. In the case where the environment is divergence-free our result can be interpreted as computing the scaling limit of the first-order correction to the quenched Central Limit Theorem.
\end{abstract}

\maketitle
\tableofcontents

\section{Introduction}
We are interested in the large-scale behavior of the following SPDE
\begin{align}\label{eq:transport}
    \partial_t \theta + \nabla\cdot(V\circ\theta) = \kappa \Delta \theta,
\end{align}
where \( V(t,x) \) is a vector Gaussian noise and $\circ$ denotes Stratonovich integration. The noise is centered, white in time, and with the following correlation structure
\begin{equation}\label{eq:Noise_corr}
     \E[V_i(t,x)V_j(s,y)]=\delta(t-s)Q_{i,j}(x-y),
\end{equation}
for an appropriate function \(Q:\R^d\rightarrow\R^{d\times d}\). Here $\delta$ means the Dirac delta function. \par

Our main interest in \eqref{eq:transport} stems from the fact that we can interpret the solution as the density of a diffusion in a random environment. Specifically, we can consider the following diffusion 
\begin{equation}\label{eq:Langragian_view}
    \dd X_t = V(t,X_t) \dd t +\sqrt{2\kappa} \dd W_t,
\end{equation}
where $W$ is a Brownian motion independent of $V$. Then, for almost every realization of \(V\), \( \theta(t,x) \) is the density of \( X_t \) with respect to the Lebesgue measure.\par

Since the vector field \( V(t,x) \) is white in time, it is not straightforward to give meaning to \eqref{eq:Langragian_view}, and therefore, the previous observation is purely formal. Nevertheless, it is possible to make this rigorous by setting up a solution theory for this SDE. This is done, for example, using Kunita's theory of stochastic flows  \cite{Kunita90}, or the theory developed in \cite{Ljan}, in the case where \( V \) has a rough correlation function. See also \cite{Dunlap-Gu} for a streamlined version of Kunita's arguments. This solution theory gives a meaning to both \eqref{eq:transport} and \eqref{eq:Langragian_view}, and makes the connection between them rigorous.\par

As such, studying the large-scale behavior of \eqref{eq:transport}, yields `local' information for a diffusion in a (white-in-time) random environment. This can be done in different scaling regimes. Specifically, \cite{PhysPaper1} distinguishes three different scaling regimes: the diffusive regime, the moderate deviation regime, and the large deviation regime. These correspond to studying the diffusion \eqref{eq:Langragian_view} under a diffusive rescaling, under a moderate tilting of the diffusion, or the large deviation behavior, respectively. Moreover, in  \cite{PhysPaper1}, the authors point out a very interesting connection to the KPZ equation and the KPZ universality class \cite{KPZref}. These conjectures were recently proved rigorously in \cite{Parekh1,Parekh2} in \(d=1\), and for the moderate deviation regime\footnote{The model studied in \cite{Parekh1,Parekh2} is a discrete analog of \eqref{eq:transport}. However, their methods also apply to the continuous case, as is pointed out in \cite{Parekh2}, Section 6.3.}. We also refer to \cite{Gu-Corwin-paper} for a study of a related, integrable model, under the large deviation regime.\par

Here, we are interested in the behavior of \eqref{eq:transport} in the diffusive scaling regime. In \cite{PhysPaper1}, it is predicted that, in this regime, the fluctuations of \(\theta(t,x)\), viewed as a random field, fall into the Edwards-Wilkinson universality class. To this end, we point out again the reference \cite{Dunlap-Gu}. There, the authors studied the point-wise behavior of \(\theta(t,x)\) and proved that 
\[
    \sup_{x\in\R^d}\E\big[ |n^{d} \theta(n^2 t, n x)- q_t(x)\Psi(n^2 t, n x)|^2 \big]\rightarrow0,
\]
as \(n\rightarrow\infty\), where \( q_t(x) \) is the standard \(d-\)dimensional heat kernel with a specific diffusivity and \( \Psi(t,x) \) is an appropriate space-time stationary random field (see \cite[Section 3]{Dunlap-Gu} for more details).  Instead, what we are interested in is the behavior of
\begin{equation}\label{eq:fluctuations_quant}
    \cX_n(t,x):=n^{d/2}(\theta^n(t,x)-\E[\theta^n(t,x)]),
\end{equation}
where \(\theta^n(t,x):=n^{d}\theta( n^2 t, n x )\), viewed as a random element of a Sobolev space with negative order. Our main result shows that \( \cX_n(t,x)\) converges in distribution to an explicit Gaussian limit, confirming the predictions of \cite{PhysPaper1}, see \textbf{Theorem \ref{thm:Main_thm}} for the precise statement. We also mention \cite{rel_paper_1, rel_paper_3, rel_paper_2} for related results on discrete models.\par
Clearly \( n^{d} \theta(n^2 t, n x) \) corresponds to the density of \( n^{-1}X_{n^2 t}\) and therefore, the result of \cite{Dunlap-Gu} corresponds to a quenched local central limit theorem for the diffusion, with a random correction due to the presence of \(\Psi\). This result yields a quenched invariance principle as well, i.e., for almost all realizations of \( V(t,x) \), \( (n^{-1}X_{n^2 t})_{t\in[0,T]} \) converges to a Brownian motion with an effective diffusivity \( D_{\rm eff} \).\par
In this context, \( \cX_n(t,x) \) can be seen as the next order correction to this invariance principle. Formally, we can write
\[
    \E[g(n^{-1}X_{n^2 t})|V]= \E[g(D_{\rm eff}B_t)]+\cR_n,
\]
where \( \cR_n \) denotes a (mean zero) random term that goes to \(0\), almost surely as \(n\rightarrow0\). Our main result is a central limit theorem for this error term (see \textbf{Remark \ref{Rem:CLT_Correction}} for more details).\par
Finally, we point out that by taking the noise \( V(t,x) \) in \eqref{eq:transport} to be a vector space-time white noise, the equation is a singular SPDE, and as such it does not make any sense\footnote{When \(V(t,x)\) has the regularity of the spacetime white noise the product \( \nabla\cdot (\theta V)\) does not make sense. }. Even worse, a formal computation, using the scaling properties of the space-time white noise, shows that it is scaling supercritical for all \( d\geq1 \), which means the theories of regularity structures \cite{Hairer} or paracontrolled distributions \cite{GIP15} cannot be used to make sense of the equation. Here, the noises we consider have better regularity, but in our scaling regime, they converge to the spacetime white noise. As such, our result can also be seen as studying the behavior of a supercritical SPDE. In fact, the  Gaussian fluctuations of \eqref{eq:fluctuations_quant} that we prove here are analogous to recent results regarding fluctuations of supercritical SPDEs (see \cite{mSHE1, mSHE2, KPZ, Burgers}).\par

We end this introduction by describing the structure of the rest of the paper. In \textbf{Section \ref{Sec:setup_assum}} we set up the solution theory for \eqref{eq:transport} and state our assumptions. We present our main result and briefly describe our methods in \textbf{Section \ref{subs-main-result}}. In \textbf{Section \ref{Sec:Cor_Fun_estimates}} we collect estimates on the correlation functions of the model and prove some a priori estimates for \eqref{eq:transport}. In \textbf{Section \ref{Sec:Proof_compr}} we prove our main result, while in \textbf{Section \ref{Sec:Proof_inc_compr}} we show that our result still holds under weaker assumptions on the correlation function of the noise, if we assume that the latter is divergence-free. 

\subsection{Notation}
\begin{itemize}
    \item For \(x\in \R^d\), we write \(x^\ast\) as the transposition of \(x\). The notation \(|x|\) stands for the usual Euclidean norm while $\langle x \rangle= (1+|x|^2)^{1/2}$.
    \item Let $\cS$ be the space of Schwartz test functions and denote the Fourier transform of a function $f:\R^d\to \R^m$ as
            \[
                \widehat{f}(\xi)= \cF(f)(\xi) = \frac1{(2\pi)^{d/2}} \int_{\R^d} {\rm e}^{-{\rm i} \xi\cdot x} f(x) \dd x, \quad \xi\in \R^d. 
            \]
        We write as usual $L^p_x= L^p(\R^d)$ the Lebesgue spaces with norm $\|\cdot \|_p= \|\cdot \|_{L^p_x},\, p\ge 1$. For $\alpha\in \R$, the notation $H^\alpha_x= H^\alpha(\R^d)$ stands for the usual inhomogeneous Sobolev space on $\R^d$, with the norm
           \[ 
                \|f\|_{H^\alpha_x}= \bigg(\int_{\R^d} \langle \xi\rangle^{2\alpha} |\widehat f(\xi)|^2 \dd\xi \bigg)^{1/2}, 
           \]
        while $\dot H^\alpha_x$ denotes the homogeneous Sobolev space where the norm is defined by replacing $\langle \xi\rangle^{2\alpha}$ with $|\xi|^{2\alpha}$. We shall adopt the same notations for spaces of vector fields on $\R^d$.
    \item For a function \( f:[0,T]\times \R^d\rightarrow\R^m\), we will use interchangeably the notation \( f_t\) and \( f(t)\) to denote the map \( x\rightarrow f(t,x)\).
    \item Given $T>0$ and $p,q\ge 1$, we denote $L^p_t L^q_x$ for the time-dependent space $L^p([0,T], L^q(\R^d))$; similarly, $C_t L^q_x$ and $C_t H^\alpha_x$ are abbreviations of $C([0,T], L^q(\R^d))$ and $C([0,T], H^\alpha(\R^d))$, respectively. Sometimes, we replace the subscript $t$ by $T$ to stress the length of the time interval $[0,T]$.
    \item We write $a\lesssim b$ to mean that there is some unimportant constant $C>0$ such that $a\le Cb$; to emphasize the dependence of $C$ on some parameters $d, \kappa$, we use the notation $a\lesssim_{d,\kappa} b$. 
    \item We will denote by \(q_t(x)=\frac{1}{(2\pi t)^{d/2}}e^{-\frac{|x|^2}{2t}}\) the standard \(d-\)dimensional heat kernel; the notation $I_d$ stands for the $d\times d$ unit matrix.
    \item We will make use of the notation \( x_{1:p}=(x_1,\dots,x_p)\) and for a function \(g:\R^d\rightarrow\R\), we define \(g^{\otimes p}:\R^{pd}\rightarrow\R\), \(g^{\otimes p}(x_{1:p}):=g(x_1)g(x_2)\dots g(x_p)\).
\end{itemize}

\subsection{The Setup and Assumptions}\label{Sec:setup_assum}
As mentioned in the introduction, we are interested in the following SPDE:
\[
    \partial_t \theta + \nabla\cdot(V\circ\theta) = \kappa \Delta \theta,
\]
where $\circ$ denotes Stratonovich integration. Before stating our main result, we first need to give a precise meaning to \eqref{eq:transport}, and to prove that this is well-posed. To do this, we will first define the noise term $V$, and write down an appropriate representation. Then we use this representation to write \eqref{eq:transport} in It\^o form, leading us to a natural notion of solution to \eqref{eq:transport}, for which we can prove existence and uniqueness. We believe that the details of these three steps are standard. Nevertheless, we write them here for the convenience of the reader.\par
Recall the correlation function of the noise in \eqref{eq:Noise_corr}.  We consider two cases: when \( Q \) is divergence-free (the incompressible case) and when \( Q \) has possibly non-zero divergence (the compressible case). For the incompressible case, we assume the following

\begin{assumption}\label{assump-covariance}
The covariance function $Q$ has a Fourier transform given by
  \begin{equation}\label{covar-Fourier}
  \widehat{Q}(\xi)= g(\xi) \bigg(I_{d} - \frac{\xi\xi^\ast}{|\xi|^2} \bigg),
  \end{equation}
where $g(\xi)= g(|\xi|)$ is a continuous nonnegative radial function satisfying $g\in (L^1\cap L^\infty)(\R^d)$. It is easy to show that $Q(0)= 2\nu I_{d}$ for some $\nu>0$.
\end{assumption}

Observe that the matrix appearing in the right-hand side of \eqref{covar-Fourier} is the projection to the subspace orthogonal to \(\xi\), so that \(Q\) is indeed divergence-free. From \eqref{covar-Fourier} one can deduce that $Q(x)$ is a symmetric matrix for any $x\in \R^d$. As $Q$ is a covariance function, it holds $Q(-x)= Q(x)^\ast =Q(x)$, i.e. $Q$ is even. \par 

If we do not wish to assume that \(Q\) is divergence-free, we instead put a stronger assumption (in terms of regularity).

\begin{assumption}\label{Assump_compressible}
    The matrix \( Q \) is smooth, even and compactly supported, such that $Q(0)= 2\nu I_{d}$.
\end{assumption}

Rigorously, one usually interprets $V$ as a cylindrical Wiener process. Specifically, we define the Hilbert space $\H$ as the completion of $C^{\infty}_c(\R\times\R^d;\R^d)$ under the following inner product
\[\langle\mathbf{f},\mathbf{g}\rangle_{\H}:=\int_\R\int_{\R^d\times\R^d}\mathbf{f}(t,x)\cdot Q(x-y)\mathbf{g}(t,y)\dd x\dd y\dd t.\]
Then over a filtered probability space $(\Omega,\cA,(\cF_t)_{t\geq0},\P)$ one views the noise $V$, as a mean zero Gaussian process $(V(\mathbf{h}))_{h\in\H}$, with covariance function given by the inner product on $\H$.\par
For our purposes, however, we will need a more refined representation of the noise. As such we interpret $V$ as the (distributional) time derivative of an appropriate $\cQ-$Wiener process on $L^2(\R^d;\R^d)$, denoted by $W_\cQ$. Here, the operator $\cQ$ is given by
\[
    \cQ f(x)=\int_{\R^d}Q(x-y)f(y)\dd y,
\]
where $f\in C^{\infty}(\R^d;\R^d)$. We refer to \cite{SPDEbook} for standard facts about Wiener processes on Hilbert spaces.\par
Observe that the operator $\cQ$ acts as a Fourier multiplier. Furthermore, since $\cQ$ is the correlation of the noise \( V(t,x) \) and satisfies either \textbf{Assumption \ref{assump-covariance}} or \textbf{Assumption \ref{Assump_compressible}}, \( \widehat{Q}(\xi) \) is positive definite for all \(\xi\in\R^d\). This implies that $\cQ^{\alpha}$ is a well-defined Fourier multiplier operator, for all $\alpha\in\R$. This allows us to define a Gaussian measure with covariance operator $\cQ$ and with Cameron-Martin space the Hilbert space $\cH:=\cQ^{1/2}L^2(\R^d;\R^d)$, equipped with the inner product
\[
    \langle \mathbf{f},\mathbf{g}\rangle:=\int_{\R^d}\cQ^{-1/2}\mathbf{f}(x)\cdot\cQ^{-1/2}\mathbf{g}(x)\dd x.
\]
It can be proved that the space $\cH$ consists of continuous, bounded vector fields. We refer to \cite[Lemma 2.2]{GalLuo-weak} and the discussion below for more details. We also record the following lemma from the same paper, see Lemma $2.3$ therein\footnote{Strictly speaking, in \cite{GalLuo-weak} the authors consider only the case where \textbf{Assumption \ref{assump-covariance}} holds, but it is easy to see that similar arguments can be used to establish the same results under \textbf{Assumption \ref{Assump_compressible}}.}:

\begin{lemma}\label{lem-Q-series}
    Let $\{\sigma_k\}_{k\in\N}$, be any orthonormal basis of $\cH$, consisting of smooth vector fields. Then 
        \[
            Q(x-y)=\sum_{n\in\N}\sigma_k(x) \sigma_k(y)^\ast,
        \]
    where the series converges absolutely and uniformly on compact sets. If \( Q \) is divergence-free, then \(\sigma_k\) is also divergence-free, for all \( k\in\N\).
\end{lemma}
Moreover, we have the following representation of the Fourier transform of \(Q\).

\begin{proposition}\label{prop:sum-fourier-sigma}
    Under Assumption \ref{assump-covariance} or \ref{Assump_compressible}, the following identity holds in the sense of distribution:
    \begin{equation*}
        \sum_{k\in\N} \widehat{\sigma}_k(\xi) \overline{\widehat{\sigma}_k(\eta)^\ast} = \widehat{Q}(\xi) \delta(\xi-\eta), \quad \xi, \eta\in \R^d,
    \end{equation*}
where the overline means complex conjugate, and $\delta$ is the Dirac delta.
\end{proposition}

\begin{proof}
    Let $\phi, \psi\in \cS(\R^d;\R^d)$. Then
    \begin{equation*}
        \begin{aligned}
            &\lim_{N\to\infty}  \iint_{\R^d\times\R^d}  \phi(\xi)^\ast \Bigg(\sum_{k=0}^N\widehat{\sigma}_k(\xi) \overline{\widehat{\sigma}_k(\eta)^\ast}\Bigg) \overline{\psi(\eta)} ~ \dd \xi \dd \eta\\
            &= \lim_{N\to\infty}  \iint_{\R^d\times\R^d}  \widehat{\phi}(x)^\ast \Bigg( \sum_{k=0}^N \sigma_k(x) \sigma_k(y)^\ast \Bigg) \overline{\widehat{\psi}(y)}~ \dd x \dd y\\
            &= \iint_{\R^d\times\R^d}  \widehat{\phi}(x)^\ast\, Q(x-y)\, \overline{\widehat{\psi}(y)}~ \dd x \dd y\\
            &= \int_{\R^d} \phi(\xi)^\ast \widehat{Q}(\xi) \frac1{(2\pi)^{d/2}} \int_{\R^d} \overline{\e^{{\rm i} y\cdot \xi} \widehat{\psi}(y)}~ \dd y \dd \xi  \\
            &= \int_{\R^d} \phi(\xi)^\ast \widehat{Q}(\xi) \overline{\psi(\xi)}~ \dd \xi,
        \end{aligned}
    \end{equation*}
    where in the second identity we used the fact that $\sum_{k=0}^N \sigma_k(x) \sigma_k(y)^\ast \to Q(x-y)$ uniformly on any compact sets.
\end{proof}

With this representation of the covariance function, we can write
\begin{equation}\label{eq:noise-series-expansion}
    W_{\cQ}(t,x)=\sum_{k\in\N}\sigma_k(x)B_k(t),
\end{equation}
where $(B_{k}(\cdot))_{k\in\N}$ is a collection of independent standard Brownian motions on $\R$, given by
\[ 
    B_k(t)=\frac{\langle W_{\cQ}(t),\cQ^{-1/2}\sigma_k\rangle}{\|\sigma_k \|_{L^2}}.
\]
Finally, going back to the noise appearing in \eqref{eq:transport}, we can write
\[
    V(t,x)=\sum_{k\in\N}\sigma_k(x) \dot B_k(t),
\]
in the sense that
\[
   V(\mathbf{h})= \sum_{k\in\N} \int_{\R_{+}} \biggl(\int_{\R^d} \sigma_k(x)\cdot \mathbf{h}(t,x) \dd x\biggr) \dd B_k(t),
\]
where $\mathbf{h}\in\H$.\par
The previous observations establish the first step listed at the beginning of this section. Now we move on to making sense of the Stratonovich integration.\par
At a formal level, the term $\nabla\cdot (V\circ \theta)$ is understood in the Stratonovich sense, namely
\[
    \partial_t \theta + \nabla\cdot(V\circ\theta) = \kappa \Delta \theta,
\]
and, using the fact that $Q(x)$ is an even function, can be written (at least formally) as It\^o integral plus correction, that is
\begin{equation}\label{eq:transport_Ito}
    \partial_t \theta + \nabla\cdot (\theta V) = \nu\Delta \theta +\kappa \Delta \theta,
\end{equation}
where $\nu$ is as in \textbf{Assumptions \ref{assump-covariance}}, \textbf{\ref{Assump_compressible}}.  This leads us to the following notion of solution to \eqref{eq:transport} (see also \cite[Definition 2.16]{GalLuo-weak}):

\begin{definition}\label{def:Ito_sol}
   Let $(\Omega,\cA,(\cF_t)_t,\P)$ be a given filtered probability space satisfying the standard assumptions, let $V$ be as above. Let $p\in (1,\infty)$ and let $\theta_0\in L^1\cap L^p$. A solution to \eqref{eq:transport} is an $(\cF_t)_t$-progressively measurable process $\theta:[0,T]\times\Omega \to L^1\cap L^p$, satisfying
    \begin{itemize}
        \item[(i)] $\theta$ is weakly continuous and in $L^\infty([0,T];L^1\cap L^p) \quad \P\text{-a.s.}$;
        \item[(ii)] For all $\phi\in C_c^\infty (\R^d)$, we have
                 \begin{align*}
                    \langle\theta_t,\phi\rangle = \langle\theta_0,\phi\rangle +\int_0^t \langle\theta_s, \nabla  \phi\cdot V(\dd s)\rangle +(\kappa+\nu)\int_0^t \langle \theta_s,\Delta\phi\rangle \dd s,
                \end{align*}
            where we interpret the stochastic It\^o integral as
            \begin{align*}
                \int_0^t \langle\theta_s, \nabla  \phi\cdot V(\dd s)\rangle = \sum_k \int_0^t \langle\nabla\phi, \sigma_k \theta_s \rangle \dd B_k(s).
            \end{align*}
    \end{itemize}
\end{definition}
Note that the above It\^o integral makes sense since, by \textbf{Lemma \ref{lem-Q-series}},
\begin{align*}
    \sum_k |\langle\nabla\phi, \sigma_k \theta_s \rangle|^2 &= \iint \theta_s(x)\nabla\phi(x) \cdot Q(x-y)\nabla\phi(y) \theta_s(y) \dd x\dd y \\
    &\le |Q(0)|\, \|\nabla \phi\|_{L^\infty}^2 \|\theta_s\|_{L^1}^2.
\end{align*}

\begin{remark}
    It can be shown, see for example \cite[Appendix B]{GalLuo-weak}, that, when the spectral intensity $g$ in \eqref{covar-Fourier} is decaying rapidly at infinity, that is, when $Q$ is sufficiently smooth, the Stratonovich formulation \eqref{eq:transport} makes sense and the equivalence with the It\^o formulation \eqref{eq:transport_Ito} holds rigorously. Hence, the point (ii) in Definition \ref{def:Ito_sol} is equivalent to:
    \begin{itemize}
        \item For  all $\phi\in C_c^\infty(\R^d)$, the process $t\rightarrow\langle\phi,\theta_t\rangle$ is a semimartingale;
        \item For all $\phi\in C_c^\infty(\R^d)$, we have
            \begin{align*}
                \langle\theta_t,\phi\rangle = \langle\theta_0,\phi\rangle +\lim_{n\rightarrow\infty}\sum_{k\leq n}\int_0^t \langle\nabla  \phi,\theta_s\sigma_k\rangle\circ \dd B_k(s) +\kappa\int_0^t \langle \theta_s,\Delta\phi\rangle \dd s.
            \end{align*}
    \end{itemize}
\end{remark}

% \begin{remark}\label{rmk:random_IC}
%     Definition \ref{def:Ito_sol} can be easily extended to the case when $\theta_0$ is a random $L^1_{loc}$ field, precisely $\theta_0\in L^1_{loc}(\R^d) \quad \P\text{-a.s.}$ and, for every $\phi\in C^\infty_c(\R^d)$, $\langle \theta_0,\phi\rangle$ is $\cF_0$-measurable.
% \end{remark}

We have the following well-posedness result:

\begin{theorem}\label{thm:Well_posed}
    Assume that \(Q\) satisfies \textbf{Assumption \ref{assump-covariance}} or \textbf{Assumption \ref{Assump_compressible}}. Then, for all \(\theta_0\in L^1\cap L^p\), \eqref{eq:transport} has a unique solution, in the sense of \textbf{Definition \ref{def:Ito_sol}}.\par Further, in the case \(Q\) satisfies \textbf{Assumption \ref{assump-covariance}}, we have the following estimates
    \begin{equation}\label{eq:inc_Lp_est}
         \quad \sup_{0\le t\le T} \|\theta_t \|_{ L^1\cap L^p} \le \|\theta_0 \|_{ L^1\cap L^p}
    \end{equation}
    and
    \begin{equation}\label{eq:energy_est}
        \quad \sup_{0\le t\le T} \|\theta_t \|_{L^2}^2 + 2\kappa \int_0^T \|\nabla\theta_t \|_{L^2}^2 \dd t \le 2 \|\theta_0 \|_{L^2}^2.
    \end{equation}
    In the case \(Q\) satisfies \textbf{Assumption \ref{Assump_compressible}}, if the initial condition $\theta_0$ is in $C^\infty_c(\R^d)$, then, for every $t>0$, the solution $\theta_t$ is also in $C^\infty_b(\R^d)$.
\end{theorem}

\begin{proof}
    In the case where \( Q \) satisfies \textbf{Assumption \ref{Assump_compressible}}, existence and uniqueness, as well as smoothness for $\theta_0\in C^\infty_c$, follow from \cite[Proposition 2.1]{Dunlap-Gu} and \cite[Theorem 3.1]{Kunita_ref_2}, the representation formula (2.5) in \cite{Dunlap-Gu} and the smoothness of the associated stochastic flow $X$ (precisely, $\P$-a.s., $X$ and its space derivatives are jointly continuous in $(t,x)$ and $DX$ is non-singular for every $(t,x)$), see e.g. \cite{Kunita84} (see also the comments before \cite[Proposition 2.1]{Dunlap-Gu}). %\todo{Mario: should be fine now; uniqueness should follows from Kunita using an embedding into a weighted negative Sobolev space} 
    In the case when \( Q \) satisfies \textbf{Assumption \ref{assump-covariance}}, existence and uniqueness follows from \cite[Theorem 1.3]{GalLuo-weak}; the estimate \eqref{eq:inc_Lp_est} follows again from \cite[Theorem 1.3]{GalLuo-weak}, while \eqref{eq:energy_est} follows similarly by taking into account Remark 3.2 therein.
    % The first assertion is deduced from \cite[Proposition 2.1]{Dunlap-Gu},\todo{for Mario: show that the solution lives in $L^1\cap L^p$ and weak continuity}
    % in the case where \( Q \) satisfies \textbf{Assumption \ref{Assump_compressible}}, or from \cite[Theorem 1.3]{GalLuo-weak} if \( Q \) satisfies \textbf{Assumption \ref{assump-covariance}}. In the latter case, the estimate \eqref{eq:inc_Lp_est} follows again from  \cite[Theorem 1.3]{GalLuo-weak}, while \eqref{eq:energy_est} follows similarly by taking into account Remark 3.2 therein. The third assertion follows from the representation formula (2.5) in \cite{Dunlap-Gu} and the smoothness of the associated stochastic flow (see also the comments before \cite[Proposition 2.1]{Dunlap-Gu}).
\end{proof}

If $\theta$ is a solution to \eqref{eq:transport}, $\P$-a.s., $\theta$ is in $L^\infty_t(H^{-d/2-\varepsilon})$ by Sobolev embedding, hence $\Delta \theta$ is in $L^\infty_t(H^{-d/2-2-\varepsilon})$. By Proposition \ref{prop:sum-fourier-sigma} (see also the proof of Proposition \ref{prop:quant_est} below), we have 
%for\todo{SK: This may need to change} $p\in (1,2]$ 
($\langle \xi\rangle:= (1+|\xi|^2)^{1/2}$)
\begin{align*}
       \sum_k\|\nabla\cdot(\theta_s\sigma_k)\|_{H^{-d/2-2}}^2
        &\le \sum_k\|\theta_s\sigma_k\|_{H^{-d/2-1}}^2 \\
        &= \sum_k \int |\widehat\theta_s \ast \widehat\sigma_k(\xi)|^2 \langle \xi\rangle^{-d-2} \dd\xi \\
        &= \int |\widehat\theta_s|^2\ast \trace \widehat Q(\xi) \langle\xi \rangle^{-d-2} \dd\xi \\
        &\lesssim \|\widehat\theta_s\|_{L^{\infty}}^2 \|{\rm Tr} \, \widehat Q \|_{L^{1}} \int \langle \xi\rangle^{-d-2} \dd\xi \\
        &\lesssim \|\theta_s\|_{L^1}^2.
    \end{align*}
    Hence $\sum_k\|\nabla\cdot(\theta_s\sigma_k)\|_{H^{-d/2-2}}^2$ is in $L^\infty_t$, $\P$-a.s., and so
    \begin{equation*}
        \int_0^t \nabla \cdot (\theta_s V(\dd s) ) = \sum_k \int_0^t \nabla \cdot (\theta_s \sigma_k) \dd B_k(s)
    \end{equation*}
    makes sense as $H^{-d/2-2}$-valued stochastic It\^o integral and is in $C_t^{\gamma}(H^{-d/2-2})$ for every $\gamma<1/2$, $\P$-a.s.. Hence \eqref{eq:transport} holds as the following SDE on $H^{-d/2-2-\varepsilon}$:
    \begin{equation*}
        \theta_t = \theta_0 -\int_0^t \nabla\cdot ( \theta_s V(\dd s)) +(\kappa+\nu)\int_0^t \Delta \theta_s \dd s.
    \end{equation*}
    In particular, $\theta$ is in $C_t^{\gamma}(H^{-d/2-2-\varepsilon})$ for every $\gamma<1/2$, $\P$-a.s.. By interpolation with the $L^\infty_t(H^{-d/2-\varepsilon})$ bound, we get that, for every $\epsilon>0$, $\theta$ is in $C_t(H^{-d/2-\epsilon})$, $\P$-a.s.\par
    These observations naturally lead us to the notion of a $H^{-d/2-\epsilon}$-mild solution to \eqref{eq:transport_Ito} (equivalently to \eqref{eq:transport}), which we will use extensively in our proofs. We say that $(\theta_t)_{t\in [0,T]}$ is a $H^{-d/2-\epsilon}$-mild solution if 
    \begin{equation}\label{eq:mild_sol}
    \theta_t=P_t\theta_0-\int_{0}^t P_{t-s}\nabla\cdot  (\theta_s V(\dd s)),
\end{equation}
a.s. in $H^{-d/2-\epsilon}$, where we interpret the It\^o integral as in \cite{SPDEbook}, $(P_t)_{t\geq0}$ be the heat semigroup generated by the operator $(\kappa+\nu) \Delta$. It is a standard fact that weak solutions to SPDEs are also mild solutions. 
%Since we are dealing with an SPDE with conservative noise, we write down the proof of this observation: \todo{SK: This needs to be done in the compressible case too}

\begin{proposition}
    Let $\epsilon>0$, \(\theta_0\in L^1\cap L^2 \) and let $\theta$ be a weak solution of \eqref{eq:transport}, as in \textbf{Definition \ref{def:Ito_sol}}, with \(\theta_0\) as the initial data. Then $\theta$ is $H^{-d/2-\epsilon}$-mild solution to \eqref{eq:transport_Ito}.
\end{proposition}

\begin{proof}
    The proof is similar to the proof from \cite[Chapter 6]{SPDEbook}. Let $(\theta_s)_{s\in[0,T]}$ be a weak solution, in the sense of Definition \ref{def:Ito_sol}. Using a density argument one can show that for all $f\in C^1([0,t];\cS(\R^d))$ we have
    \[
         \langle\theta_t,f_t\rangle- \langle\theta_0,f_0\rangle 
         =\int_0^t \langle\theta_s, \nabla  f_s\cdot V(\dd s)\rangle +\int_0^t \big\langle \theta_s, \partial_s f_s+(\kappa+\nu)\Delta f_s \big\rangle \dd s.
    \]
    We choose $f_s=P_{t-s}\phi$, where $\phi\in C_c^\infty(\R^d)$. Since $\partial_s f_s+(\kappa+\nu)\Delta f_s=0$, we get
    \[
        \langle\theta_t,\phi\rangle-\langle\theta_0,P_t\phi\rangle=-\int_0^t\langle P_{t-s}\phi , \nabla\cdot (\theta_s V(\dd s)) \rangle.
    \]
    A straightforward adaptation of the proof of \textbf{Proposition \ref{prop:quant_est}} below,  shows that  
    \[
        \int_0^t P_{t-s} \nabla\cdot (\theta_s V(\dd s)),
    \]
    is in $H^{-d/2-\epsilon}$, when $\{\theta_s\}_{s\geq0}$ is a progressively measurable process such that 
    \[
        \sup_{s\in[0,t]}\E\big[\|\theta_s\|_{L^2_x}^2 \big]<\infty.
    \]
    The latter is true by \eqref{eq:inc_Lp_est}, when \( Q \) satisfies \textbf{Assumption \ref{assump-covariance}}. If \( Q \) satisfies \textbf{Assumption \ref{Assump_compressible}} this bound can be obtained by adapting the proof of \textbf{Lemma \ref{lemm:Lp_moment_bound}} (see also \textbf{Remark \ref{rem:quant_est_unscaled}}).\par
    As such, we can write
    \[
        \langle\theta_t,\phi\rangle-\langle P_t\theta_0,\phi\rangle=- \bigg\langle\int_0^t P_{t-s}\nabla\cdot (\theta_s V(\dd s)),\phi \bigg\rangle,
    \]
    where we also used the fact that $P_t$ is self-adjoint. This identity holds a.s. for all $\phi\in C_c^\infty(\R^d)$. By the density of $C_c^\infty(\R^d)$ in $H^{-d/2-\epsilon}$ (recall that $\theta$ is in $C_t(H^{-d/2-\epsilon})$ by our observations before the statement of the proposition), we conclude that \eqref{eq:mild_sol} holds a.s. in $H^{-d/2-\epsilon}$. 
\end{proof}

Finally, we point out that, in the case where \( Q \) satisfies \textbf{Assumption \ref{Assump_compressible}}, \cite[Proposition 3.1 and Corollary 3.2]{Dunlap-Gu} construct a space-time stationary solution\footnote{By stationary field we mean a field whose law is invariant under space-time translations $f\mapsto f(\cdot+s,\cdot+y)$. Technically, to define $\Psi$ as solution to \eqref{eq:transport} in our setting, we should extend Definition \ref{def:Ito_sol} to include negative times and data in $L^1_{loc}$. While this is possible (as in \cite{Dunlap-Gu}), in this paper we only need that $\Psi$ satisfies \eqref{eq:DG_result_compressible} and \eqref{converg-second-moment}.} to \eqref{eq:transport}, which we denote by \(\Psi(t,x)\), such that \(\E[\Psi(t,x)]=1\), and \(\E[\Psi(t,x)^2]<\infty\). The correlation function of this field will appear in the statement of our main result. As such, define \(V_{\rm eff}\) to be the symmetric matrix, such that 
\begin{equation}\label{eq:eff_var}
    V_{\rm eff}^2=\int_{\R^d}\E[\Psi(0,0)\Psi(0,z)]Q(z)\dd z.
\end{equation}
Observe that, since $Q\in L^1$ and \(\Psi(t,z)\) has a finite second moment, this integral is finite.

\subsection{Main Result and Outline of the proof}\label{subs-main-result} 
To study the fluctuations of \eqref{eq:transport}, we introduce a parameter $n\in \N$ (which we will send to $\infty$). First we rescale the initial condition to  \eqref{eq:transport}:
\[
    \theta(0,x)=n^{-d} \varphi(x/n),
\]
where $\varphi(x)$ is a smooth and compactly supported function, independent of $n$. Let $\theta$ be the corresponding solution to \eqref{eq:transport}, which is understood in the It\^o form \eqref{eq:transport_Ito} and is well-posed by Theorem \ref{thm:Well_posed} (of course $\theta_0$ and so $\theta$ depend on $n$, but, with a small abuse of notation, we will omit this dependence, to make notation lighter). We then consider the diffusive scaling for $\theta$:
\begin{equation}\label{eq:densityrescaled}
   \theta^n(t,x)=n^{d}\theta(n^2t, nx).
\end{equation}
Note that $\theta^n$ satisfies \eqref{eq:transport_Ito}, with initial condition $\theta^n(0,x)=\varphi(x)$, but with $V$ replaced by $V^n(t,x) = nV(n^2t,nx)$. In particular, the mean of $\theta^n$ solves the heat equation  
\begin{equation}\label{eq:heat}
   \partial_t \bar\theta= (\kappa+\nu) \Delta \bar\theta, \quad \bar\theta(0)= \varphi.
\end{equation}
With these assumptions, \eqref{eq:fluctuations_quant} is written as
\begin{equation}\label{eq:chi_def}
    \cX_n(t,x):=n^{d/2}(\theta^n(t,x)-\bar\theta(t,x)).
\end{equation}
Our main result shows that \(\cX_n(t,x) \) has a Gaussian limit, as expected by the CLT-type scaling:

\begin{theorem}\label{thm:Main_thm}
      Let $\alpha> d/2$ and $\gamma\in (0,1/2)$. Assume that \( Q \) satisfies \textbf{Assumption \ref{Assump_compressible}}. Then, for every $\varphi\in C^{\infty}_c(\R^d)$, $\cX_n(t,x)$, as defined in \eqref{eq:chi_def}, converges in distribution in $C^\gamma([0,T]; H^{-\alpha}_{\rm loc})$ to  $\cU(t,x)$,  the solution of the following additive stochastic heat equation:
    \begin{equation}\label{eq:limit_eq}
        \partial_t\cU=(\kappa+\nu)\Delta\cU+ \nabla\cdot(\bar\theta\, V_{\rm eff} \xi),\quad \cU(0,x)=0,
    \end{equation}
    where $\xi$ is a vector-valued space-time white noise, and \( V_{\rm eff}\) as in \eqref{eq:eff_var}.\par
    In the case where \( Q \) satisfies \textbf{Assumption \ref{assump-covariance}}, the same is true, where the limiting equation is given by \eqref{eq:limit_eq} with \(V_{\rm eff}^2= g(0) \Pi\), where $\Pi$ is the Helmholtz-Leray projection. 
\end{theorem}

\begin{remark}\label{rem:Q_assum}
    The field \(\Psi(t,x)\) is equal to \(1\) iff \(Q\) is divergence-free, see \cite{Dunlap-Gu}; if $Q$ is also integrable, then a slight modification of the proof in Section \ref{Sec:Proof_inc_compr} shows that $V_{\rm eff}^2= \int_{\R^d} Q(z)\dd z$ in this case, thus the first part of the theorem covers the case where the correlation function is divergence-free and integrable. The point is that the divergence-free condition allows for the control of the SPDE, \eqref{eq:transport}, even if the correlation function of the noise is not smooth, or compactly supported. In fact, \textbf{Assumption \ref{assump-covariance}} allows for Kraichnan-type noises, precisely we can take \(g(\xi)=\langle \xi\rangle^{-(d+\zeta)} \), for \(\zeta\in(0,2)\). On the other hand, it seems possible \textbf{Assumption \ref{Assump_compressible}} is too strong. Indeed, the results of \textbf{Section \ref{Sec:Cor_Fun_estimates}} can be proved under the assumption that \( Q \) is twice continuously differentiable, see \textbf{Remark \ref{rem:weaker_assum}}. Here, we work under \textbf{Assumption \ref{Assump_compressible}} in order to use the results of \cite{Dunlap-Gu}, specifically, the estimate \eqref{eq:DG_result_compressible} below. We believe that \eqref{eq:DG_result_compressible} remains true under weaker assumptions on \( Q \), but we do not pursue this here in order not to detract from our main result.
\end{remark}

\begin{remark}
    If one assumes that the correlation function \( Q \) is divergence-free, \textbf{Assumption \ref{assump-covariance}} is nearly optimal. Indeed, if one restricts even further and takes \(g(\xi)=\langle \xi\rangle^{-(d+\zeta)} \) as in the previous remark, then we cannot prove a statement as in \textbf{Theorem \ref{thm:Main_thm}} when \( \zeta=0 \). In this case, the It\^o-Stratonovich correction is infinite and \eqref{eq:transport} does not make sense. This is still interesting, however, as in \( d=2\) it formally corresponds to a white in time noise that spatially 'looks like' the \(2d\) Gaussian Free Field. With this choice of noise, \eqref{eq:transport} is formally scaling critical, so in analogy to recent works on critical SPDEs \cite{Burgers, Nikos_1, Nikos_2, Simon_paper, Dun_paper, XM_paper} we expect that a logarithmic correction to the scaling considered here is needed in order to see non-trivial behavior.
\end{remark}

\begin{remark}\label{Rem:CLT_Correction}
    Observe that we take as the initial data \( \varphi\in C_c^\infty(\R^d) \). When \(\varphi\) is nonnegative and \( \| \varphi \|_{L^1}=1 \), we can interpret this choice as starting the diffusion \eqref{eq:Langragian_view} with an initial condition \( X_0\sim n^{-d}\varphi(x/n) \dd x\).
    %\todo{Mario: initial law of $X_0$ should be rescaled, hence I replaced $\varphi(x)$ with $n^{-d}\varphi(x/n)$} 
    Then, we can write 
    \[
       \langle\cX_n(t,\cdot),g\rangle= n^{d/2}\big(\E_{\varphi}[g(n^{-1}X_{n^2 t})|V]-\E[g(\tilde B_t)]\big),
    \]
    where \(\E_{\varphi}\) is the expectation with respect to the law of the diffusion \eqref{eq:Langragian_view}, with \(\varphi(x)\dd x\) as the initial distribution and conditional on \(V\), and \((\tilde B_t)_{t\geq0}\)  denotes a Brownian motion with diffusivity \((\kappa+\nu) I_{d}\), with \( \varphi(x)\dd x\) as an initial distribution.\par
    The result of \cite{Dunlap-Gu}, implies that 
    \[
        \E_{\varphi}[g(n^{-1}X_{n^2 t})|V]-\E[g(\tilde B_t)] \rightarrow0,
    \]
    almost surely. This is a form of a quenched central limit theorem. Therefore, as mentioned in the introduction, \textbf{Theorem \ref{thm:Main_thm}} shows that the scaling limit of the first order correction to the quenched CLT is Gaussian with an explicit variance. Notably, we cannot take \(\varphi\) to be a Dirac delta function centered at \(0\) (i.e., start \(X_t\) from \( 0\)). We believe this to be an artifact of the proof, and similar methods can be used to extend our main results in this case as well. 
\end{remark}

Let us sketch the basic idea of the proof. We make use of the linearity of \eqref{eq:transport_Ito} (more exactly, its rescaled version \eqref{eq-n}) and the white-in-time correlations of the noise to write \eqref{eq:chi_def} as a stochastic integral. In particular,  \eqref{eq:chi_def} is a martingale. Therefore, to prove \textbf{Theorem \ref{thm:Main_thm}}, we show that \eqref{eq:chi_def} is tight in $C^\gamma([0,T]; H^{-\alpha}_{\rm loc})$, and that its quadratic variation converges to the quadratic variation of the martingale part of the mild solution of \eqref{eq:limit_eq}. This, combined with the Skorohod representation theorem, will allow us to conclude.\par
To prove tightness for \eqref{eq:chi_def}, we rely on quantitative estimates  for moments of 
\[
    \|\theta_t^n - \bar \theta_t\|_{\dot H^{-\alpha}_x}.
\]
This is the content of \textbf{Proposition \ref{prop:quant_est}}. To prove these estimates, we need to control the moments of \( L^2\) norms of \(\theta_t^n\), which is done in \textbf{Section \ref{Sec:Cor_Fun_estimates}}. This control is immediate in the case where the noise is divergence-free, as we can use \eqref{eq:inc_Lp_est}. To obtain a similar control under \textbf{Assumption \ref{Assump_compressible}}, we rely on the correlation functions of \eqref{eq:transport_Ito}. These functions satisfy a closed-form PDE that has a fundamental solution which, in turn, satisfies appropriate heat kernel bounds (see \textbf{Propositions \ref{prop:Cor_PDE_derivation}} and \textbf{\ref{prop:corr_bound}}). Using these observations, we can obtain the required control of the \(L^2\) moments of \(\theta_t^n\), see \textbf{Lemma \ref{lemm:Lp_moment_bound}}.

Having these estimates at hand, we can prove tightness of the laws of \( \{\cX_n\}_{n\ge 1} \) in $C^\gamma([0,T]; H^{-\alpha}_{\rm loc})$, see \textbf{Lemma \ref{lem-continuity}}  and \textbf{Proposition \ref{prop-tightness-X-n}}. To calculate the limiting covariance, we make use of the pointwise limiting statistics of \( \theta^n(t,x)\). In particular, under \textbf{Assumption \ref{Assump_compressible}}, we make use of the result of \cite{Dunlap-Gu}. More specifically, for a fixed \(\eps>0\), we use the bound from \cite[Theorem 1.1]{Dunlap-Gu}\footnote{Actually, in \cite{Dunlap-Gu} this is proved in the case where $\varphi$ is a delta function, centered at $0$, and for all \(t\geq t_0\), for some \(t_0>0\). A straightforward adaptation of the arguments in \cite{Dunlap-Gu} can show \eqref{eq:DG_result_compressible} as well.}
\begin{equation}\label{eq:DG_result_compressible}
    \sup_{x\in\R^d} \E\big[|\theta^n(t,x)- q_t\ast \varphi(x)\Psi(n^2t,nx)|^2 \big]\lesssim_{\eps} n^{-\gamma}
\end{equation}
for some $\gamma>0$, and all $t\geq \eps$, where $q_t$ is the standard heat kernel on $\R^d$.\par

\section{Proofs}\label{Sec:Proofs}

\subsection{Correlation Functions and a priori estimates}\label{Sec:Cor_Fun_estimates}

Throughout this section, we work under \textbf{Assumption \ref{Assump_compressible}}. We consider \eqref{eq:transport_Ito}, with initial data $\theta_0 \in C_c^\infty(\R^d)$. It will become apparent that to control moments of the $H^{-\alpha}_x$ norm of $\theta_t-\bar \theta_t$, we will need to control $L^p_x$ norms of $\theta(t,x)$. In particular, we seek to prove a bound of the form 
\begin{equation}\label{eq:ideal_bound}
    \sup_{t\in[0,T]}\E\big[\|\theta_t\|_{L^p_x}^{r} \big]\lesssim 1.
\end{equation}
To prove this, we utilize the correlation functions of the model \eqref{eq:transport_Ito}.  More specifically, for any $p\in\N$, we define the $p$-th correlation function
\begin{equation}\label{eq:p_point_corr}
    \cS_p(t,x_{1:p}):=\E[\theta(t,x_1)\dots\theta(t,x_p)],
\end{equation}
where we recall the notation $x_{1:p}=(x_1,\dots,x_p)$. As we will see,  good pointwise bounds for the correlation functions imply bounds of the form \eqref{eq:ideal_bound} (see \textbf{Proposition \ref{prop:corr_bound}} and \textbf{Lemma \ref{lemm:Lp_moment_bound}}, below). As mentioned in the previous section, the advantage of dealing with the correlation functions, instead of $\|\theta\|_{L^p_x}$ directly, is that \eqref{eq:p_point_corr} satisfies an explicit parabolic PDE. Indeed, define the matrix
\begin{equation}\label{eq:p_eddy_diff_matrix}
    \cC_p(x_{1:p})= \begin{pmatrix}
                        (\kappa+\nu) I_d && Q(x_1-x_2)^\ast && \dots && Q(x_1- x_p)^\ast\\
                        Q(x_2-x_1) && (\kappa+\nu) I_d && \ldots && Q(x_2-x_p)^\ast \\
                        \vdots && \vdots && \ddots && \vdots\\
                        Q(x_p-x_1)      && Q(x_p-x_2) && \dots &&  (\kappa+\nu) I_d
                    \end{pmatrix}.
\end{equation}
We have the following proposition.

\begin{proposition}\label{prop:Cor_PDE_derivation}
    For all $p\in\N$, the correlation function \eqref{eq:p_point_corr} is a weak solution to 
    \begin{equation}\label{eq:p_point_corr_PDE}
\partial_t\cS_p(t,x_{1:p})= \trace(\nabla^2(\cC_p(x_{1:p}) \cS_p(t,x_{1:p}))),
     \end{equation}
     with $\cS_p(0,x_{1:p})=\theta_0^{\otimes p}(x_{1:p})$, where $\theta_0\in C^\infty_c(\R^d)$ is the initial data for \eqref{eq:transport}.
\end{proposition}
Here, for a matrix $A$, we denote
\[
    \trace(\nabla^2(Af)):=\sum_{i,j=1}\partial_{i,j}^2(a_{i,j}f)
\]

\begin{proof}
    
Let \(g_1,\dots, g_p\in C^\infty_c(\R^d)\). From \textbf{Definition \ref{def:Ito_sol}} we have
    \begin{equation}
        \langle\theta_t,g_i\rangle= \langle\theta_0,g_i\rangle+\int_0^t\langle\theta_s,(\kappa+\nu)\Delta g_i\rangle \dd s-\int_0^t\langle\nabla g_i, \theta_s V(\dd s)\rangle,
    \end{equation}
    for all \(i=1,\dots,p\).
    Now, for
    \[
        f(z_1,\dots,z_p):=\prod_{i=1}^pz_i,
    \]
    we apply It\^o's formula on $f(\langle\theta_t,g_1\rangle,\dots,\langle\theta_t,g_p\rangle)$. This yields
    \begin{equation}\label{eq:Ito_for_app}
        \dd f(\langle\theta_t,g_1\rangle,\dots,\langle\theta_t,g_p\rangle)=\nabla f\cdot \dd\Theta_t^{(p)} + \frac{1}{2}\sum_{i\neq j} \partial_{x_i, x_j}f \dd[\langle\theta_t,g_i\rangle, \langle\theta_t,g_j\rangle]_t,
    \end{equation}
    where $\Theta^{(p)}_t:=(\langle\theta_t,g_1\rangle,\dots,\langle\theta_t,g_p\rangle)$. The first term on the right-hand side is equal to 
    \begin{align}\label{eq:first_corr_term}
       \sum_{i=1}^p \partial_{x_i} f \dd \langle\theta_t,g_i\rangle=\sum_{i=1}^p\prod_{\substack{j=1,\\ j\neq i}}^p \langle\theta_t,g_j\rangle \bigl(\langle\theta_t,(\kappa+\nu)\Delta g_i\rangle \dd t-\dd\langle\nabla g_i, \theta_tV\rangle\bigr).
    \end{align}
    On the other hand, the second term on the right-hand side of \eqref{eq:Ito_for_app} is equal to 
    \begin{align}\label{eq:second_cor_term}
        &\sum_{i<j}\prod_{\substack{m=1,\\ m\neq i,j}}\langle\theta_t,g_m\rangle\dd[\langle\theta_t,g_i\rangle,\langle\theta_t,g_j\rangle]_t\nonumber\\
        &=\sum_{i<j}\prod_{\substack{m=1,\\ m\neq i,j}}\langle\theta_t,g_m\rangle\biggl(\int_{\R^{2d}}\nabla g_i(x_1)Q(x_1-x_2)\nabla g_j(x_2) \theta_t(x_{1})\theta_t(x_{2})\dd x_{1:2}\biggr)\dd t,
    \end{align}
    where we used the fact that
    \[\aligned
        &\dd[\langle\nabla g_i, \theta_tV(\dd t)\rangle, \langle\nabla g_j, \theta_tV(\dd t)\rangle]_t \\
        &=\biggl(\int_{\R^{2d}}\nabla g_i(x_1)Q(x_1-x_2)\nabla g_j(x_2) \theta_t(x_{1})\theta_t(x_{2})\dd x_{1:2}\biggr)\dd t,
    \endaligned\]
    Plugging \eqref{eq:first_corr_term} and \eqref{eq:second_cor_term} to \eqref{eq:Ito_for_app}, and then taking the expectation, shows that 
    \[
        \langle\cS_p(t,\cdot),G\rangle=\langle\theta_0^{\otimes p}, G\rangle-\int_0^t\langle\cS_p(s,\cdot), \cC_p\nabla^2G\rangle\dd s,
    \]
    where \(G(x_{1:p})=g_1(x_1)g_2(x_2)\dots g_p(x_p)\). Arguing by density (as $\theta_t$ is smooth in space by the last assertion of Theorem \ref{thm:Well_posed}) concludes the proof.
\end{proof}

The point is that the PDE \eqref{eq:p_point_corr_PDE} is well-behaved, as the next proposition shows.

\begin{proposition}\label{prop:corr_bound}
    The PDE \eqref{eq:p_point_corr_PDE} has a fundamental solution, which we denote by $G_p(t,y_{1:p}, x_{1:p})$. We also have the following heat kernel bound 
    \begin{equation}\label{eq:pth_hkernel_bound}
        G_p(t,y_{1:p},x_{1:p})\lesssim q^{\otimes p}_{ct}((x-y)_{1:p}),
    \end{equation}
    where we recall that $q_t(x)$ is the standard $d$-dimensional heat kernel and $c>0$ is a constant.
\end{proposition}

\begin{proof}
    Observe that the matrix \( \cC_p(x_{1:p})\) is uniformly elliptic and, from \textbf{Assumption \ref{Assump_compressible}} it is also smooth. From standard results in parabolic PDEs, \cite{PDE_ref}, this implies the existence of \( G_p(t,y_{1:p}, x_{1:p})\). To prove \eqref{eq:pth_hkernel_bound}, we observe that the adjoint problem associated to \eqref{eq:p_point_corr_PDE}:
    \[
        \partial_t f =\trace(\cC_p\nabla^2 f),
    \]
    is a non-divergence form parabolic PDE with smooth, uniformly elliptic coefficients. Similarly to before, this PDE has a fundamental solution \(\tilde G_{p}(t,y_{1:p}, x_{1:p})\). From \cite[Remark 5.12]{Heat_k_ref},  \( \tilde G_p\) satisfies the bound \eqref{eq:pth_hkernel_bound}. By noticing that \(G_p(t,y_{1:p},x_{1:p})=\tilde G_p(t,x_{1:p},y_{1:p})\), we conclude the proof.
\end{proof}

\begin{remark}\label{rem:weaker_assum}
    This bound is used  to prove \textbf{Lemma \ref{lemm:Lp_moment_bound}}, \textbf{Lemma \ref{lemm:resc_corr_bound}} and \textbf{Proposition \ref{prop:quant_est}}. As such, it is one of the central ingredients of our arguments. We note that this proposition holds under weaker assumptions. Indeed \cite{Heat_k_ref} requires the coefficients to have finite Dini mean oscillation. On the other hand, to make sense of the equation using Kunita's theory, one convenient assumption would be to take \( Q \) to be twice continuously differentiable. As such, we expect that we can prove \textbf{Theorem \ref{thm:Main_thm}} under the latter assumption. As mentioned in \textbf{Remark \ref{rem:Q_assum}}, we do not pursue this here to avoid re-proving statements we need from \cite{Dunlap-Gu}, which are proved there under \textbf{Assumption \ref{Assump_compressible}}.
\end{remark}

With the bound \eqref{eq:pth_hkernel_bound}, one can estimate moments of the $L^2$ norm of $\theta_t$. For our purposes, we will need to estimate moments of the rescaled solutions $\theta^n$. In particular, we have the following lemma

\begin{lemma}\label{lemm:Lp_moment_bound}
   Let $\varphi\in C_c^\infty(\R^d)$, and start equation \eqref{eq:transport_Ito} with $\theta_0(x)=n^{-d}\varphi(x/n)$ as the initial data. Recall that $\theta^n(t,x)=n^d\theta(n^2t,nx)$. Then for all $r\in\N$, we have
    \[ 
    \sup_{t\in[0,T]} \E\big[\|\theta^n(t) \|_{L^2_x}^{2r} \big]^{1/2r} \lesssim\|\theta_0 \|_{L^2}.
    \]
\end{lemma}

\begin{proof}
    Observe that the expectation in the statement of the lemma is equal to
    \[
        \int_{\R^{dr}} \E\bigg[\prod_{i=1}^{r} |\theta_t^n(x_i)|^2 \bigg] \dd x_{1:r}.
    \]
    This motivates us to consider the correlation function $\cS_{2r}(t,y_{1:2r})$. Using the fundamental solution of \eqref{eq:p_point_corr_PDE}, we can write
    \[
        \cS_{2r}(t,y_{1:2r})=\int_{\R^{2dr}} G_{2r}(t,y_{1:2r},z_{1:2r}) \theta_0^{\otimes 2r}(z_{1:2r})\dd z_{1:2r}.
    \]
    We use the bound \eqref{eq:pth_hkernel_bound} which yields
    \[
        |\cS_{2r}(t,y_{1:2r})|\lesssim\int_{\R^{2dr}} q_{ct}^{\otimes 2r}(y_{1:2r}-z_{1:2r})| \theta_0^{\otimes 2r}(z_{1:2r})|\dd z_{1:2r}.
    \]
    This holds for all $y_{1:2r}\in\R^{2dr}$. We choose
    \[
        y_1=y_2=x_1,\quad y_3=y_4=x_2,\dots, y_{2r-1}= y_{2r} = x_{r}.
    \]
    With this choice, and recalling that \(\theta_0(z)=n^{-d}\varphi(z/n)\) we have
    \[\aligned
        \E\bigg[\prod_{i=1}^{r}|\theta_t(x_i)|^2 \bigg] 
        &= |\cS_{2r}(t,y_{1:2r}) |\lesssim \prod_{i=1}^{r} q_{ct}\ast|\theta_0|(x_i)^2 \\
        &=\prod_{i=1}^{r} \bigg(\int_{\R^d}q_{ct}(x_i-nz)|\varphi(z)|\dd z \bigg)^2.
    \endaligned\]
    Taking \( t\rightarrow n^2 t\), \( x_i\rightarrow n x_i\), multiplying both sides by $n^{2dr}$, and integrating over $x_{1:r}$, yields the bound
    \[
        \E\big[\|\theta^n(t) \|_{L^2_x}^{2r} \big]=\int_{\R^{dr}} \E\bigg[\prod_{i=1}^{r} |\theta_t^n(x_i)|^2 \bigg] \dd x_{1:r}\lesssim\|\varphi\|_{L^2},
    \]
    where we also used the scaling properties of the standard heat kernel and the bound $\|q_t\ast \varphi\|_{L^2_x}\leq \|\varphi\|_{L^2_x}$. This concludes the proof.
\end{proof}

\begin{remark}\label{rem:quant_est_unscaled}
    More generally, the same argument works to bound moments of the \( L^2_x \) norm of a solution to \eqref{eq:transport_Ito} without rescaling. In particular, if \( \theta_t \) solves \eqref{eq:transport_Ito} with \( \varphi \in L^2(\R^d)\) as the initial data, then 
    \[
        \sup_{t\in[0,T]} \E\big[\|\theta(t) \|_{L^2_x}^{2r} \big]^{1/2r} \lesssim\|\theta_0 \|_{L^2}
    \]
    One could use a similar argument to derive bounds on \( \E[\| \theta_t\|_{L^p_x}^{pr}]\), for \(p,r\in\N\), but we do not pursue this here.
\end{remark}

We are also going to need a pointwise bound on the correlation function. This is given by the following lemma. 

\begin{lemma}\label{lemm:resc_corr_bound}
Under the same setting as in the previous lemma, we have
\[
    |\cS^n_{2}(t,x,y)|=|\E[\theta^n(t,x) \theta^n(t,y)]| \lesssim q_{ct+1}(x) q_{ct+1}(y).
\]
\end{lemma}

\begin{proof}
The proof is very similar to the proof of the previous lemma. We can write 
\[
    \cS_2(t,x,y)=\int_{\R^{2d}}G_2(t,x,y;z_1,z_2)n^{-d}\varphi(z_1/n)n^{-d}\varphi(z_2/n)\dd z_1 \dd z_2.
\]
From \eqref{eq:pth_hkernel_bound}, we get
\[
    |\cS_2(t,x,y)|\lesssim\int_{\R^{2d}} q_{ct}(x-z_1) q_{ct}(y-z_2)|n^{-d}\varphi(z_1/n)n^{-d}\varphi(z_2/n)| \dd z_1 \dd z_2.
\]
Since $\varphi$ is compactly supported, we have $|n^{-d}\varphi(z_1/n)|\lesssim n^{-d} q_1(z_1/n)= q_{n^2}(z_1)$. Therefore, we get the bound
\[
   |\cS_2(t,x,y)|\lesssim q_{ct+ n^2}(x) q_{ct+ n^2}(y).
\]
Again, taking $t\rightarrow n^2t$, $(x,y)\rightarrow (nx,ny)$, and multiplying both sides by $n^{2d}$ concludes the proof.
\end{proof}

\subsection{Proof of \textbf{Theorem \ref{thm:Main_thm}}: The compressible case}\label{Sec:Proof_compr}

Here, we prove \textbf{Theorem \ref{thm:Main_thm}}, under the \textbf{Assumption \ref{Assump_compressible}}. As mentioned in the introduction, the idea is the following: First, we will prove that \((\cX_n)_n\), defined in \eqref{eq:chi_def}, is tight and then we will characterize the law of all limiting points. 

Recall that $\theta$ solves equation \eqref{eq:transport_Ito} with initial condition $\theta(0,x)=n^{-d} \varphi(x/n)$; $\theta^n$, defined by $\theta^n(t,x) = n^d \theta(n^2 t, nx)$, satisfies the equation
\begin{equation}\label{eq-n}
    \partial_t\theta^n +\nabla\cdot (\theta^n V^n) = (\kappa+\nu) \Delta \theta^n,
\end{equation}
with the rescaled initial condition
\[
\theta^n(0,x)=\varphi(x),
\]
which is independent of $n$. Here $V^n$ in \eqref{eq-n} is defined by $V^n(t,x) = nV(n^2t,nx)$
%$\theta(0,x)= n^{-d} \varphi(n^{-1}x)$. Therefore, the initial condition for \eqref{eq-n} is 
%  \[\theta^n(0,x)= n^d \theta(0, n x)= \varphi(x) \]
%which is independent of $n$.The noise $V^n$ in \eqref{eq-n}
and has spatial covariance function $Q^n(x)= Q(nx),\, x\in \R^d$; therefore, it has the Fourier transform
\begin{align}\label{Fourier-Q-n}
  \widehat{Q^n}(\xi)= n^{-d} \widehat{Q}(n^{-1}\xi).
\end{align}
Note that $V^n= \dot W^n$ and thus by \eqref{eq:noise-series-expansion},
\begin{align*}
  V^n(t,x)= n V(n^2 t, nx) = \sum_{k=1}^\infty \sigma_k(nx)\, n \dot B_k(n^2 t);
\end{align*}
as a result,
\begin{align}\label{eq-noise-n}
  W^n(t,x)= \sum_{k=1}^\infty \sigma_k(nx)\, n^{-1} B_k(n^2 t) \stackrel{\mathcal L}{=} \sum_{k=1}^\infty \sigma_k(nx) B_k(t)
\end{align}
by the scaling property of Brownian motions. Finally, note that $Q^n(0)= Q(0)= 2\nu I_{d}$ is independent of $n$.

To prove \textbf{Theorem \ref{thm:Main_thm}}, we first prove an estimate for the moments of negative Sobolev norms of the solution to \eqref{eq-n}. Similar estimates appear in \cite{LXZ24}, see also \cite{FGL24} for the estimates in the case of torus. The difference here is that we do not assume that the correlation function $Q$ is divergence-free. Recall that \(\bar\theta_t(x)\) solves \eqref{eq:heat} with \( \varphi \) as the initial data.

\begin{proposition}\label{prop:quant_est}
    For all $\alpha\in(\frac{d}{2},\frac{d}{2}+1)$, and all $q\in\N$, we have
    \begin{equation}\label{converg-theta-n}
            \E \Big[ \sup_{t\in [0,T]} \|\theta^n_t - \bar \theta_t\|_{\dot H^{-\alpha}_x}^{q} \Big]^{1/q}
                \lesssim_{d, \alpha, q, T} \|\varphi\|_{2} \big\| \widehat{Q} \big\|_{\infty}^{1/2} n^{-d/2},
    \end{equation}
    In particular,
     \[ 
        \E \Big[ \sup_{t\in [0,T]} \|\cX_n(t) \|_{\dot H^{-\alpha}_x}^{q} \Big]^{1/q} \lesssim_{d, \alpha, q } \|\varphi\|_{2} \big\| \widehat{Q} \big\|_{\infty}^{1/2} .
    \]
\end{proposition}

\begin{proof}
    Recall that $P_t= \e^{(\kappa+ \nu)t \Delta},\, t\geq 0$ is the heat semigroup; using the mild form of \eqref{eq-n} and of the heat equation \eqref{eq:heat}, we see that the difference $\theta_t^n - \bar \theta_t$ is nothing but the stochastic convolution
\begin{equation*}
    Z_t^n := -\sum_{k} \int_0^t P_{t-r}\nabla\cdot(\theta_r^n\sigma_k^n) \dd B_k^n(r) ,
\end{equation*}
where $\sigma^n_k(x)= \sigma_k(nx)$, $B^n_k(r)= n^{-1} B_k(n^2 r)$, $k\ge 1$
We have
  \begin{equation}\label{proof-stoch-convol-1}
  \aligned
  \big[ \E\|Z_t^n\|_{\dot H^{-\alpha}}^{2q} \big]^{1/q}
  &= \bigg[ \E \Big\| \sum_k \int_0^t P_{t-r}\nabla\cdot (\theta_r^n \sigma_k ^n) \dd B_k^n(r) \Big\|_{\dot H^{-\alpha}}^{2q} \bigg]^{1/q} \\
  &\lesssim_q \bigg[ \E \Big( \sum_k \int_0^t \big\| P_{t-r}\nabla\cdot (\theta_r ^n\sigma_k^n ) \big\|_{\dot H^{-\alpha}}^2 \dd r \Big)^q  \bigg]^{1/q} ,
  \endaligned
  \end{equation}
where in the second step we have used the Burkholder-Davis-Gundy inequality in the Hilbert space $\dot H^{-\alpha}$. Noting that
  $$\aligned
  \big\| P_{t-r}\nabla\cdot (\theta_r^n \sigma_k^n ) \big\|_{\dot H^{-\alpha}}^2
  &= \int_{\R^d} |\xi|^{-2\alpha} \big|\mathcal F\big( P_{t-r}\nabla\cdot (\theta_r^n \sigma_k^n ) \big)(\xi) \big|^2 \dd\xi \\
  &= \int_{\R^d} |\xi|^{-2\alpha} {\rm e}^{-2(\kappa+\nu) |\xi|^2(t-r) } \big|\xi\cdot \mathcal F(\theta_r^n \sigma_k^n )(\xi) \big|^2 \dd\xi \\
  &= (2\pi)^{-d} \int_{\R^d} |\xi|^{-2\alpha} {\rm e}^{-2(\kappa+\nu) |\xi|^2 (t-r)} \big|\xi\cdot \big( \widehat{\theta^n}(r)\ast \widehat{\sigma^n_k} \big)(\xi) \big|^2 \dd\xi ,
  \endaligned $$
therefore,
  $$\aligned
  &\big[ \E\|Z_t^n\|_{\dot H^{-\alpha}}^{2q} \big]^{1/q} \\
  &\lesssim_{d,q} \bigg[ \E \Big( \sum_k \int_0^t\! \int_{\R^d} |\xi|^{-2\alpha} {\rm e}^{-2(\kappa+\nu) |\xi|^2 (t-r)} \big|\xi\cdot \big( \widehat{\theta^n}(r)\ast \widehat{\sigma^n_k} \big)(\xi) \big|^2 \dd\xi \dd r \Big)^q  \bigg]^{1/q} \\
  &= \bigg[ \E \Big(  \int_{\R^d} |\xi|^{-2\alpha} \int_0^t {\rm e}^{-2(\kappa+\nu) |\xi|^2 (t-r)} \sum_k \big|\xi\cdot \big( \widehat{\theta^n}(r)\ast \widehat{\sigma^n_k} \big)(\xi) \big|^2 \dd r \dd\xi \Big)^q \bigg]^{1/q}.
  \endaligned $$
By Proposition \ref{prop:sum-fourier-sigma}, and the scaling properties of the Fourier transform, we see that 
  \begin{equation}\label{proof-stoch-convol-2}
  \begin{aligned}
      &\sum_k \big| \xi\cdot (\widehat{\theta^n}(r)\ast \widehat{\sigma^n_k})(\xi) \big|^2 \\
      &=\lim_{N\to\infty} \sum_{k=0}^N \iint_{\R^d\times\R^d} \widehat{\theta^n}(r, \xi-\eta)~ \xi^\ast \widehat{\sigma^n_k}(\eta) \overline{\widehat{\sigma^n_k} (\zeta)^\ast} \xi ~ \overline{\widehat{\theta^n}(r, \xi-\zeta)} \dd \eta \dd \zeta \\
      &= \int_{\R^d} \xi^\ast\widehat{Q^n}(\eta) \xi\, \big| \widehat{\theta^n}(r,\xi-\eta) \big|^2 \dd \eta.
      \end{aligned}
  \end{equation}
Plugging this into the inequality above, we get
  \begin{equation*}
  \big[ \E\|Z_t^n\|_{\dot H^{-\alpha}}^{2q} \big]^{1/q} \lesssim \bigg[ \E \Big( \int_{\R^d} \! |\xi|^{-2\alpha}\! \int_0^t \! {\rm e}^{-2(\kappa+\nu) |\xi|^2 (t-r)} \xi^\ast \big(|\widehat{\theta^n}(r)|^2 \ast \widehat{Q^n} \big)(\xi) \xi \dd r \dd\xi \Big)^q \bigg]^{1/q} .
  \end{equation*}
 Using Young's inequality, we get
    \[
        \big|\big(|\widehat{\theta^n}(r)|^2\ast\widehat{Q^n}\big)(\xi) \big| \leq \big\| |\widehat{\theta^n}(r)|^2 \big\|_{1} \big\|\widehat{Q^n} \big\|_{\infty} = \|\widehat{\theta^n}(r)\|_{2}^2 \big\|\widehat{Q^n} \big\|_{\infty} =\|\theta^n_r\|_{2}^2 \big\|\widehat{Q^n} \big\|_{\infty}.
    \]
    Plugging this into the previous bound yields
    \[
        \big[ \E\|Z_t^n\|_{\dot H^{-\alpha}}^{2q} \big]^{1/q} \lesssim \bigg[ \E \Big( \int_0^t\!\! \int_{\R^d} \! |\xi|^{-2\alpha+2} {\rm e}^{-2(\kappa+\nu) |\xi|^2 (t-r)}  \|\theta^n_r\|_{2}^2 \big\|\widehat{Q^n} \big\|_{\infty} \dd r\dd\xi \Big)^q  \bigg]^{1/q}.
    \]
    Now, from Minkowski's inequality, we get
    \[
        \big[ \E\|Z_t^n\|_{\dot H^{-\alpha}}^{2q} \big]^{1/q} \lesssim \big\|\widehat{Q^n} \big\|_{\infty} \int_0^t\!\! \int_{\R^d} \! |\xi|^{-2\alpha+2} {\rm e}^{-2(\kappa+\nu) |\xi|^2 (t-r)}  \E\big[\|\theta^n_r\|_{2}^{2q} \big]^{1/q} \dd r\dd\xi.
    \]
    By \textbf{Lemma \ref{lemm:Lp_moment_bound}}, we have
$$\aligned
   \big[ \E\|Z_t^n\|_{\dot H^{-\alpha}}^{2q} \big]^{1/q}
   &\lesssim \|\varphi \|_{L^2_x}^2 \big\|\widehat{Q^n} \big\|_{\infty} \int_0^t \!\! \int_{\R^d} \! |\xi|^{-2\alpha+2} {\rm e}^{-2(\kappa+\nu) |\xi|^2 (t-r)} \dd r\dd\xi \\
   &=  \|\varphi \|_{L^2_x}^2 \big\|\widehat{Q^n} \big\|_{\infty} \int_{\R^d} |\xi|^{-2\alpha} \frac{1- {\rm e}^{-2(\kappa+\nu) |\xi|^2 t}}{2(\kappa+\nu)} \dd\xi  \\
   &\le  \|\varphi \|_{L^2_x}^2 \big\|\widehat{Q^n} \big\|_{\infty} \int_{\R^d} |\xi|^{-2\alpha} \big[(2\nu)^{-1} \wedge (|\xi|^2 t) \big] \dd\xi .
  \endaligned $$
Using spherical coordinates, one has
  $$\aligned
  & \int_{\R^d} |\xi|^{-2\alpha} \big[(2\nu)^{-1}\wedge (|\xi|^2 t) \big] \dd\xi \\
  &= c_d \int_0^\infty \rho^{-2\alpha} \big[(2\nu)^{-1}\wedge (\rho^2 t) \big] \rho^{d-1} \dd\rho \\
  &\lesssim_{d, \alpha} t \int_0^{(2\nu t)^{-1/2}} \rho^{-2\alpha + d+1} \,\dd\rho + (2\nu)^{-1} \int_{(2\nu t)^{-1/2}}^\infty \rho^{-2\alpha + d-1} \dd\rho \\
  &\lesssim_{d, \alpha, \nu} t^{\alpha-d/2},
  \endaligned $$
thanks to the constraint $\alpha\in \big(\frac d2, 1+\frac d2 \big)$. Substituting this estimate into the above inequality yields
\[\big[ \E\|Z_t^n\|_{\dot H^{-\alpha}}^{q} \big]^{1/q} \leq \big[ \E\|Z_t^n\|_{\dot H^{-\alpha}}^{2q} \big]^{1/2q} \lesssim_{q,d, \alpha} \|\theta_0 \|_{L^2_x} \big\|\widehat{Q} \big\|_{\infty}^{1/2} t^{(2\alpha -d)/4}. \]

It remains to show that the above estimate can be improved by inserting $\sup_{t\in [0,T]}$ in the expectation. This can be done in the same way as the end of proof of \cite[Lemma 3.1]{LXZ24}. Hence, we omit the details here.
\end{proof}

\begin{remark}\label{rem:quant_est_divfree}
    In the case where we do not rescale the solution and the initial data, we can show the following estimate
    \begin{equation}\label{eq:quant_bound}
        \E\Big[\sup_{t\in [0,T]} \|\theta_t-\bar\theta_t\|_{\dot H^{-\alpha}_x}^{q}\Big]^{1/q} \lesssim_{d,\alpha, q, T} \|\varphi\|_{2} \big\|\widehat{Q}\big\|_{\infty}^{1/2},
    \end{equation}  
    where \( \theta_t \) solves \eqref{eq:transport} with \( \varphi \) as the initial data. Indeed, we can follow the exact same computations, with the difference that we use the bound in \textbf{Remark \ref{rem:quant_est_unscaled}} instead of \textbf{Lemma \ref{lemm:Lp_moment_bound}}.  Notice that \eqref{eq:quant_bound} has the same form as \eqref{converg-theta-n}, the only difference coming from the rescaling of the correlation function of the noise. 
    
    As mentioned before, estimate \eqref{eq:quant_bound} is very similar to the ones in \cite[Theorem 1.5]{LXZ24}, with the major technical difference being that we do not assume that \( Q \) is divergence-free. Nevertheless, a similar bound to \eqref{eq:quant_bound} holds under \textbf{Assumption \ref{assump-covariance}} as well. This can be seen by following the same arguments as in the proof of \eqref{eq:quant_bound} but using \eqref{eq:inc_Lp_est} in place of \textbf{Lemma \ref{lemm:Lp_moment_bound}}. In fact, more generally, under \textbf{Assumption \ref{assump-covariance}} we have
    \begin{equation}\label{eq:gen_quant_est_div_free}
        \E \Big[ \sup_{t\in [0,T]} \|\theta_t - \bar \theta_t\|_{\dot H^{-\alpha}_x}^{q} \Big]^{1/q} \lesssim_{d, \alpha, \nu, q, T} \|\theta_0\|_{p} \big\|\widehat{Q} \big\|_{p/(2-p)}^{1/2}, 
    \end{equation}
    for \(p\in(1,2]\), \(q\geq1\) and $\alpha\in(\frac{d}{2},\frac{d}{2}+1)$. The proof \eqref{eq:gen_quant_est_div_free} is similar to the proof of \textbf{Proposition \ref{prop:quant_est}}, and since we are not going to use this bound, we skip the details, cf. \cite[Theorem 1.5]{LXZ24}.
\end{remark}

Now, note that $\cX_n$ has the expression
\begin{equation}\label{eq:fluctuations-cX-n}
\aligned 
  \cX_n(t) &= n^{d/2} \int_0^t P_{t-r}\nabla\cdot (\theta^n_r V^n(\dd r)) \\
  &= n^{d/2} \sum_{k} \int_0^t P_{t-r} \nabla\cdot (\theta^n_r \sigma^n_k ) \dd B^n_k(r),
\endaligned 
\end{equation}
where $\sigma^n_k(x)= \sigma_k(nx)$, $B^n_k(r)= n^{-1} B_k(n^2 r)$, $k\ge 1$. We remark that for any fixed $n\ge 1$, $\{B^n_k \}_{k\ge 1}$ are mutually independent. We turn to showing that $\cX_n$ converges in the weak sense.

\begin{lemma}\label{lem-continuity}
Let $\alpha>d/2$ and $\delta>0$ be such that $\alpha-\delta>d/2$, then for any $q\in2\N$ and $0\le s<t\le T$, we have
\[ 
  \E\|\cX_n(t) -\cX_n(s)\|_{H^{-\alpha}}^q \lesssim \|\varphi \|_2^q \big\|\widehat{Q} \big\|_\infty^{q/2} |t-s|^{\delta q/2}. 
\]  
\end{lemma}
Before moving on to the proof, we recall the following two elementary inequalities. Let \(\alpha\in\R\), \(\rho\geq0\). Then 
\begin{equation}\label{eq:Sobolev_smoothing}
    \|P_tg\|_{H^{\alpha+\rho}}\lesssim t^{-\rho/2}\|g\|_{H^\alpha},
\end{equation}
and for \(\rho\in[0,2]\)
\begin{equation}\label{eq:Sobolev_smoothing_centered}
    \|(P_t-I)g\|_{H^{\alpha-\rho}}\lesssim t^{\rho/2}\|g\|_{H^\alpha},
\end{equation}
where \( I \) denotes the identity operator.
\begin{proof}
The computations below are similar to the proof of Proposition \ref{prop:quant_est}. By \eqref{eq:fluctuations-cX-n}, we have
\[\aligned
  \cX_n(t) -\cX_n(s) &= n^{d/2} \int_0^s \sum_k (P_{t-r}- P_{s-r}) (\nabla\cdot(\sigma^n_k \theta^n_r))\dd B^n_k(r) \\
  &\quad + n^{d/2} \int_s^t \sum_k P_{t-r}(\nabla\cdot(\sigma^n_k  \theta^n_r))\dd B^n_k(r) \\
  &=: I^n_1 + I^n_2.
\endaligned \]
By Burkholder's inequality in Hilbert space $H^{-\alpha}$, \eqref{eq:Sobolev_smoothing} and \eqref{eq:Sobolev_smoothing_centered} we have
\[\aligned
  \E \|I^n_1 \|_{H^{-\alpha}}^q &\lesssim_q n^{qd/2} \E\bigg(\int_0^s \sum_k \big\|(P_{t-s}-I ) P_{s-r}(\nabla\cdot(\sigma^n_k  \theta^n_r)) \big\|_{H^{-\alpha}}^2 \dd r\bigg)^{q/2} \\
  &\lesssim n^{qd/2} \E\bigg(\int_0^s \sum_k |t-s|^\delta \big\|P_{s-r}(\nabla\cdot(\sigma^n_k  \theta^n_r)) \big\|_{H^{-\alpha+\delta}}^2 \dd r\bigg)^{q/2} \\
  &\lesssim_\nu (n^{d} |t-s|^{\delta})^{q/2} \E\bigg(\int_0^s \sum_k \frac 1{|s-r|^{1-\eps}} \big\|\nabla\cdot(\sigma^n_k  \theta^n_r) \big\|_{H^{-\alpha+\delta-1+\eps}}^2 \dd r\bigg)^{q/2} \\
  &\lesssim (n^{d} |t-s|^{\delta})^{q/2} \E\bigg(\int_0^s \frac 1{|s-r|^{1-\eps}} \sum_k\big\|\sigma^n_k \theta^n_r \big\|_{H^{-\alpha+ \delta+ \eps}}^2 \dd r\bigg)^{q/2},
\endaligned \]
where $\eps$ is small enough, and in the third step we have used the semigroup property. We have
\[\aligned
  \sum_k\big\|\sigma^n_k \theta^n_r \big\|_{H^{-\alpha+ \delta+ \eps}}^2 
  &= \sum_k \int_{\R^d} \<\xi \>^{-2(\alpha-\delta -\eps)} \big|\cF(\theta^n_r\sigma^n_k) (\xi)\big|^2 \dd\xi \\
  &=\frac1{(2\pi)^d} \int_{\R^d} \<\xi \>^{-2(\alpha-\delta -\eps)} \sum_k \big|\big( \widehat{\theta^n}(r)\ast \widehat{\sigma^n_k} \big)(\xi) \big|^2 \dd \xi.
\endaligned\]
Similarly to the proof of \eqref{proof-stoch-convol-2}, one has
\[\aligned
   \sum_k \big|\big( \widehat{\theta^n}(r)\ast \widehat{\sigma^n_k} \big)(\xi) \big|^2 
   &= \sum_k \iint_{\R^d\times \R^d} \widehat{\theta}_n(r,\xi-\eta) \overline{\widehat{\theta}_n(r,\xi-\zeta)} \widehat{\sigma^n_k}(\eta) \cdot \overline{\widehat{\sigma^n_k}(\zeta)} \dd\eta \dd\zeta \\
   &= \int_{\R^d} \big|\widehat{\theta}_n(r,\xi-\eta)\big|^2 {\rm Tr}(\widehat{Q^n}(\eta)) \dd\eta .
\endaligned \]
Therefore, 
\[\aligned
  \sum_k\big\|\sigma^n_k \theta^n_r \big\|_{H^{-\alpha+ \delta+ \eps}}^2 
  &= \frac1{(2\pi)^d} \int_{\R^d} \<\xi \>^{-2(\alpha-\delta -\eps)} \Big( \big|\widehat{\theta^n}(r) \big|^2 \ast {\rm Tr}\big(\widehat{Q^n} \big) \Big)(\xi) \dd\xi.
\endaligned\]
By Young's inequality, for any $\xi\in \R^d$, 
\[
  \Big( \big|\widehat{\theta^n}(r) \big|^2 \ast {\rm Tr}\big(\widehat{Q^n} \big) \Big)(\xi) 
  \le \Big\| \big|\widehat{\theta^n}(r) \big|^2 \Big\|_{1}  \big\|{\rm Tr}\big(\widehat{Q^n} \big) \big\|_{\infty} 
  \le \big\| \widehat{\theta^n}(r) \big\|_{2}^2 \big\| \widehat{Q^n} \big\|_{\infty}.  
\]
Recall that $\widehat{Q^n}= n^{-d} \widehat{Q}(n^{-1}\cdot)$, thus $\big\| \widehat{Q^n} \big\|_{\infty}= n^{-d} \big\| \widehat{Q} \big\|_{\infty}$; as a result,
\[
  \Big( \big|\widehat{\theta^n}(r) \big|^2 \ast {\rm Tr}\big(\widehat{Q^n} \big) \Big)(\xi) \le \big\| \widehat{\theta^n}(r) \big\|_{2}^2\, n^{-d} \big\| \widehat{Q} \big\|_{\infty} = n^{-d} \big\| \widehat{Q} \big\|_{\infty} \| \theta^n_r \|_{2}^2. 
\]
Substituting this estimate into the above equality, we arrive at
\[\aligned
  \sum_k\big\|\sigma^n_k \theta^n_r \big\|_{H^{-\alpha+ \delta+ \eps}}^2 
  &\le n^{-d} \big\| \widehat{Q} \big\|_{\infty} \| \theta^n_r \|_{2}^2 \int_{\R^d} \<\xi \>^{-2(\alpha-\delta -\eps)} \dd\xi \\
  &\lesssim n^{-d} \big\| \widehat{Q} \big\|_{\infty}  \| \theta^n_r \|_{2}^2,
\endaligned\]  
where we have used the fact that the integral is finite for $\alpha-\delta -\eps> d/2$; this is possible by taking $\eps$ small enough since $\alpha-\delta> d/2$. To sum up, we arrive at
\[\aligned
  \E \|I^n_1 \|_{H^{-\alpha}}^q &\lesssim_q (n^{d} |t-s|^{\delta})^{q/2} \E\bigg(\int_0^s \frac 1{|s-r|^{1-\eps}} n^{-d} \big\| \widehat{Q} \big\|_{\infty} \| \theta^n_r \|_{2}^2 \dd r\bigg)^{q/2} \\
  &\lesssim_q (n^{d} |t-s|^{\delta})^{q/2} \bigg(\int_0^s \frac 1{|s-r|^{1-\eps}} n^{-d} \big\| \widehat{Q} \big\|_{\infty} \E\big[ \| \theta^n_r \|_{2}^{q} \big]^{2/q} \dd r\bigg)^{q/2} \\
  &\lesssim_{\eps, T} \big\| \widehat{Q} \big\|_{\infty}^{q/2} \|\varphi \|_{2}^q |t-s|^{\delta q/2},
\endaligned \]
where we have used Minkowski's inequality and \textbf{Lemma \ref{lemm:Lp_moment_bound}}.
%\todo{Mario: also here Lemma 2.4 is used for $\theta^n$} 

Next, we estimate $I^n_2$: again by Burkholder's inequality and \eqref{eq:Sobolev_smoothing},
  \[\aligned
  \E \|I^n_2 \|_{H^{-\alpha}}^q &\lesssim_q n^{dq/2} \E\bigg(\int_s^t  \sum_k \frac 1{|t-r|^{1-\delta}} \big\| \nabla\cdot(\sigma^n_k  \theta^n_r)\big\|_{H^{-\alpha-1+\delta}}^2 \dd r\bigg)^{q/2} \\
  &\lesssim n^{dq/2} \E\bigg(\int_s^t \sum_k \frac 1{|t-r|^{1-\delta}} \big\| \sigma^n_k \theta^n_r \big\|_{H^{-\alpha+\delta}}^2 \dd r\bigg)^{q/2} .
  \endaligned \]
Repeating the above calculations, we have
  \[\aligned
  \sum_k \big\| \sigma^n_k \theta^n_r \big\|_{H^{-\alpha+\delta}}^2 
  &\lesssim \int_{\R^d} \<\xi \>^{-2(\alpha-\delta)} \Big( \big|\widehat{\theta^n}(r) \big|^2 \ast {\rm Tr}\big(\widehat{Q^n} \big) \Big)(\xi) \dd\xi \\
  &\lesssim_{\alpha,\delta} n^{-d} \big\| \widehat{Q} \big\|_{\infty} \| \theta^n_r \|_{2}^2;
  \endaligned \]
as a result,
  \[\aligned
  \E \|I^n_2 \|_{H^{-\alpha}}^q &\lesssim \big\| \widehat{Q} \big\|_{\infty}^{q/2} \|\varphi \|_{2}^q\bigg(\int_s^t \frac 1{|t-r|^{1-\delta}} \E\big[\| \theta^n_r \|_{2}^{q} \big]^{2/q} \dd r\bigg)^{q/2} \\
  &\lesssim_\delta \big\| \widehat{Q} \big\|_{\infty}^{q/2} \|\varphi \|_{2}^q |t-s|^{\delta q/2}.
  \endaligned \]  
Combining the above two estimates, we finish the proof.
\end{proof}

\begin{proposition}\label{prop-tightness-X-n}
For any $\alpha>d/2$ and $\gamma\in (0,1/2)$, the laws of $\{\cX_n\}$ are tight in $C^\gamma([0,T]; H^{-\alpha}_{\rm loc})$.
\end{proposition}

\begin{proof}
Taking $q$ big enough in the estimate of Lemma \ref{lem-continuity}, we conclude from Kolmogorov's modification theorem that $\{\cX_n(t) \}_{0\le t\le T}$ has a version which is $\P$-a.s. $\gamma$-H\"older continuous in $H^{-\alpha}$; we still denote this version by $\cX_n$.

Moreover, we have
  \[ \sup_n \E \|\cX_n \|_{C^\gamma_t H^{-\alpha}_x}^q \lesssim C<\infty,
  \]
thus for any $\eps>0$, we deduce from \cite[Corollary A.5]{GalLuo-weak} that the laws of $\{\cX_n \}_n$ are tight in $C^\gamma([0,T]; H^{-\alpha-\eps}_{\rm loc})$. By the arbitrariness of $\alpha>d/2$ and $\eps>0$, we conclude the desired assertion. 
\end{proof}

Next, we turn to characterizing the law of all limiting points. Observe that $\cX_n$ satisfies the equation
  \[\dd \cX_n(t) = (\kappa+ \nu)\Delta \cX_n(t) \dd t - n^{d/2}\nabla \cdot( V^n(\dd t) \theta^n_t); \]
we denote the  martingale part by
  \[M_n(t) = n^{d/2} \int_0^t \nabla\cdot(\theta^n(s)V^n(\dd s)) = n^{d/2} \sum_k \int_0^t \nabla\cdot(\sigma^n_k \theta^n(s)) \dd B^n_k(s). \]
Similarly to Lemma \ref{lem-continuity}, for $\beta> 1+d/2$ and small $\delta>0$,  we can show that, for any $q\ge 1$,
  \[\E\|M_n(t)- M_n(s)\|_{H^{-\beta}}^q \lesssim \|\varphi \|_2^q \big\|\widehat Q \big\|_\infty^{q/2} |t-s|^{\delta q/2}; \]
as a result, we have the following analog of Proposition \ref{prop-tightness-X-n}.

\begin{proposition}\label{prop-tightness-M-n}
The laws of martingales $\{M_n\}_n$ are tight in $C^\gamma([0,T]; H^{-\beta}_{\rm loc})$ for any $\beta>1+d/2$ and $\gamma\in (0,1/2)$. 
\end{proposition}

We fix $\alpha>d/2$, $\gamma\in (0,1/2)$ and $\beta>1+d/2$ and we consider the quantity $\Xi_n= \big(\cX_n, M_n, \theta^n, \{B_k\}_k,\Psi\big)$%\todo{Mario: added $\Psi$: the reason is that we need to copy also $\Psi$ on the same space of $\tilde\theta^n$, to use \eqref{eq:eff_var}} 
which takes values in the product space
  \[\mathcal S= C^\gamma([0,T]; H^{-\alpha}_{\rm loc}) \times C^\gamma([0,T]; H^{-\beta}_{\rm loc}) \times C([0,T]; H^{-\alpha}) \times C([0,T];\R)^{\N} \times L^2_{loc}([0,T]\times \R^d). \]
Combining \eqref{converg-theta-n} with Propositions \ref{prop-tightness-X-n} and \ref{prop-tightness-M-n}, we conclude that the laws $\{\mu_n \}_n$ of the family $\{\Xi_n \}_n$ are tight on $\mathcal S$. By Prokhorov's theorem, there is a subsequence, still denoted by $\{\mu_n \}_n$ for simplicity, converging weakly, in the topology of $\mathcal S$, to some limit probability measure $\mu$. Then Skorohod's representation theorem (see e.g. \cite[Chapter 3, Theorem 1.8]{Kurtz86}) implies that there exist a new (complete) probability space $\big(\tilde\Omega, \tilde\F, \tilde\P \big)$ and stochastic processes $\tilde \Xi_n= \big(\tilde \cX_n, \tilde M_n, \tilde \theta^n, \{\tilde B^n_k \}_k, \tilde\Psi_n \big)$, and a limit process $\tilde \Xi= \big(\tilde \cX, \tilde M, \tilde \theta, \{\tilde B_k \}_k, \tilde\Psi \big)$ such that 
\begin{itemize}
\item[(a)] $\tilde \Xi= \big(\tilde \cX, \tilde M, \tilde \theta, \{\tilde B_{k} \}_k, \tilde\Psi \big) \stackrel{\mathcal L}{\sim} \mu$ and $\tilde \Xi_n= \big(\tilde \cX_n, \tilde M_n, \tilde \theta^n, \{\tilde B^n_k \}_k, \tilde\Psi_n \big) \stackrel{\mathcal L}{\sim} \mu_n$  for any $n\ge1$;
\item[(b)] $\tilde\P$-a.s., $\tilde \Xi_n$ converges in the topology of $\mathcal S$ to $\tilde \Xi$ as $n\to \infty$.
\end{itemize}
For each $n$, we define the filtration $(\tilde\F^n_t)_t$ as follows: for each $t$, $\tilde\F^{0,n}_t$ is the $\sigma$-algebra generated by $\sigma\big(\tilde \theta^n(s), \{\tilde B_k^n(s)\}_k: s\le t \big)$ and by the $\tilde\P$-null sets, and we take $\tilde\F^n_t:=\cap_{t'>t} \tilde\F^{0,n}_{t'}$. We define the filtrations $(\tilde\F^0_t)_t$ and $(\tilde\F_t)_t$ for the limiting processes analogously using $\tilde \theta(s), \{\tilde B_k\}_k$. It is a standard fact (see e.g. the proof of \cite[Proposition 2.5, point (1)]{Bas2011}) that $(\tilde\F^n_t)_t$ are complete and right-continuous, $\tilde\theta^n$ are $(\tilde\F^n_t)_t$-progressively measurable and $\{\tilde B^n_k\}_k$ are independent $(\tilde\F^n_t)_t$-Brownian motions; similarly for the limiting process and filtration.

Item (a) above implies that $\tilde\theta= \bar\theta$ solves the deterministic heat equation \eqref{eq:heat}; moreover, for any $n\ge 1$, the following stochastic equation holds in the distributional sense:
\[\aligned
  \dd \tilde \cX_n(t)&= (\kappa+\nu) \Delta\tilde \cX_n(t) \dd t - \dd \tilde M_n(t) \\
  &= (\kappa+\nu) \Delta\tilde \cX_n(t) \dd t - n^{d/2} \sum_k \nabla\cdot (\sigma^n_k \tilde \theta^n(s)) \dd \tilde B^n_k(t) .
\endaligned \]
In particular, for any $\phi\in C_c^\infty(\R^d)$, $\tilde\P$-a.s. for all $t\in [0,T]$, we have
\[
  \<\tilde \cX_n(t), \phi\> = (\kappa+\nu) \int_0^t \<\tilde \cX_n(s), \Delta\phi\> \dd s - \<\tilde M_n(t), \phi\>; 
\]
by item (b) above, letting $n\to \infty$ yields
\[
  \<\tilde \cX(t), \phi\> = (\kappa+\nu) \int_0^t \<\tilde \cX(s), \Delta\phi\> \dd s - \<\tilde M(t), \phi\>. 
\]  

It remains to identify the limit object $\{\tilde M(t) \}_{t\in [0,T]}$. In particular, we need to show that \(\{\tilde M(t) \}_{t\in [0,T]}\) is a Gaussian martingale, with the correct quadratic variation. We have 

\begin{lemma}\label{lem:mart_proof}
    The process $\{\tilde M_t\}_{t\in[0,T]}$ is an $H^{-\beta}$-valued martingale w.r.t. the filtration \(\tilde \F_t\).
    %\(\tilde \F_t= \sigma\big(\tilde \theta(s), \{\tilde B_k(s)\}_k: s\le t \big)\). 
\end{lemma}

\begin{proof}
%Let $\tilde \F^n_t= \sigma\big(\tilde \theta^n(s), \{\tilde B^n_k(s) \}_k: s\le t \big)$.
We know that $\tilde M_n$ is an $H^{-\beta}$-valued continuous martingale w.r.t. $\{\tilde \F^n_t \}_t$; thus for any $0\le s<t\le T$, $\phi\in C_c^\infty (\R^d)$, and any bounded continuous functional $G: C([0,s], H^{-\zeta})\times C([0,s], \R^{\N}) \to \R$, we have
\[
  \tilde \E\big[ \big\<\tilde M_n(t) -\tilde M_n(s), \phi \big\>\, G\big(\tilde \theta^n (\cdot), \{\tilde B^n_k(\cdot) \}_k \big) \big] = 0. 
\] 
Letting $n\to\infty$ we obtain
\[
  \tilde \E\big[ \big\<\tilde M(t) -\tilde M(s), \phi \big\>\, G\big(\tilde \theta_\cdot, \{\tilde B_k(\cdot) \}_k \big) \big] = 0. 
\]
By the arbitrariness of $0\le s< t\le T$, $\phi\in C_c^\infty (\R^d)$ and the functional $G$, we deduce that $\tilde M$ is an $H^{-\beta}$-valued martingale w.r.t. $\{\tilde \F^0_t \}_t$ and hence also w.r.t. $\{\tilde \F _t \}_t$.
\end{proof}

To show that $\tilde M$ is a Gaussian process, we need to identify its quadratic variation; in particular, we need to calculate the limit of the quadratic variation of $\langle \tilde M_n(t),\phi\rangle$, where $\phi\in C_c^\infty(\R^d)$. Note that by item (a) above, for any $\phi\in C_c^\infty(\R^d)$, we have
\[
    \tilde M_n(t) = n^{d/2} \sum_k \int_0^t \nabla\cdot(\sigma^n_k \tilde \theta^n_s) \dd\tilde B^n_k(s),
\]
where we have written $\tilde\theta^n(s)$ as $\tilde\theta^n_s$ to save space. We have

\begin{lemma}\label{lem:quad_var_lim}
    As $n\rightarrow\infty$, we have that
    \[
        [\langle \tilde M_n,\phi\rangle](t) \rightarrow\int_0^t \!\! \int_{\R^{d}} \big( q_{r}\ast \varphi(x) \big)^2 \nabla \phi(x)^\ast V_{\rm eff}^2 \nabla \phi(x) \dd x \dd r ,
    \]
    in probability. In particular, the quadratic variation of $\langle \tilde M_n(t),\phi\rangle$ converges to the quadratic variation of 
    \[
        \int_0^t\langle\nabla\phi,V_{\rm eff}\, \bar\theta_s\, \xi(\dd s) \rangle,
    \]
    where $\xi:=(\xi_1,\dots,\xi_d)$, with $\xi_i$ being the standard space time white noise on $\R^d$.
\end{lemma}

\begin{proof}
We calculate
    \[\aligned 
  [\<\tilde M_n, \phi\>](t) &= n^d \sum_k \int_0^t \big\< \tilde \theta^n_s, \sigma^n_k \cdot\nabla \phi \big\>^2 \dd s \\
  &= n^d \sum_k \int_0^t\!\! \iint (\tilde \theta^n_s \nabla \phi)(x)^\ast \sigma^n_k(x)\otimes \sigma^n_k(y) (\tilde \theta^n_s \nabla \phi)(y) \dd x \dd y \dd s \\
  &= n^d \int_0^t\!\! \iint (\tilde \theta^n_s \nabla \phi)(x)^\ast Q^n (x-y) (\tilde \theta^n_s \nabla \phi)(y) \dd x \dd y \dd s ,
\endaligned \]
where the superscript $\ast$ means transposition of vectors and $Q^n(z)= Q(nz),\, z\in \R^d$. Writing $Q_n(z)=n^dQ(nz)= n^d Q^n(z)$ yields the expression 
\begin{equation}\label{eq-quadratic-var}
    [\langle\tilde M_n, \phi\rangle](t)=\int_0^t\!\! \int_{\R^{2d}} \tilde \theta^n_s(x)\tilde \theta^n_s (y) \nabla \phi(x)^\ast Q_n (x-y) \nabla \phi(y) \dd x \dd y \dd s.
\end{equation}
We split the time integral over the intervals $[0,\eps]$ and $[\eps,t]$, and call the corresponding terms $T_1$ and $T_2$ respectively. 

First, we show that $\tilde\E[|T_1|]=O(\eps)$. Indeed we have
\[
    \tilde\E[|T_1|]\leq\int_0^\eps\!\! \int_{\R^{2d}} \E\big[ |\tilde \theta^n_s(x)\tilde \theta^n_s (y)| \big] |\nabla \phi(x)^\ast Q_n (x-y) \nabla \phi(y)| \dd x \dd y \dd s.
\]
Now, we use item (a) and \textbf{Lemma \ref{lemm:resc_corr_bound}} to get the bound 
\[
   \tilde\E[|T_1|]\leq\int_0^\eps\!\! \int_{\R^{2d}} q_{cs+1}(x) q_{cs+1}(y) |\nabla \phi(x)^\ast Q_n (x-y) \nabla \phi(y)| \dd x \dd y \dd s=O(\eps),
\]
where we used the fact that the  integral 
\[
    \int_{\R^{2d}} |\nabla \phi(x)^\ast Q_n (x-y) \nabla \phi(y)| \dd x \dd y \leq \|Q_n \|_{L^1} \|\nabla\phi \|_{L^2} ^2 = \|Q \|_{L^1} \|\nabla\phi \|_{L^2} ^2
\]
is finite, as $n\rightarrow\infty$.

Since \(\eps\) is arbitrarily small,  we only need to handle the term $T_2$. Recall that we denote by \(\Psi(t,x)\) the space-time stationary solution to \eqref{eq:transport_Ito}, constructed in \cite{Dunlap-Gu}. We observe that, using the estimate in Lemma \ref{lemm:Lp_moment_bound},
%\todo{Mario: are we using also Lemma \ref{lemm:Lp_moment_bound}, to control the term $\theta^n_s(x)$?}
\[
    \tilde\E \bigg[ \bigg| T_2-\int_\eps^t\!\! \int_{\R^{2d}} \tilde \theta^n_s(x) q_{s}\ast\varphi(y)\tilde\Psi_n(n^2s,ny) \nabla \phi(x)^\ast Q_n (x-y) \nabla \phi(y) \dd x \dd y \dd s \bigg| \bigg]\rightarrow0,
\]
as $n\rightarrow\infty$. Here, we made use of \eqref{eq:DG_result_compressible}. Let us denote by $\tilde T_2$ the second term in the above expression. It is easy to see that
\begin{align*}
    \tilde\E \bigg[\bigg| \tilde T_2-\int_\eps^t\!\! \int_{\R^{2d}}  q_{s}\ast\varphi(x)\tilde\Psi_n(n^2s,nx) q_{s}\ast\varphi(y)\tilde\Psi_n(n^2s,ny) \\ 
    \nabla \phi(x)^\ast Q_n (x-y) \nabla \phi(y) \dd x \dd y \dd s\bigg|\bigg] \rightarrow 0.
\end{align*}
This shows that we need to consider
\[
    \cT_2=\int_\eps^t\!\! \int_{\R^{2d}} \tilde\Psi_n^n(s,x)\tilde\Psi_n^n(s,y) q_{s}\ast\varphi(x) q_{s}\ast\varphi(y) \nabla \phi(x)^\ast Q_n (x-y) \nabla \phi(y) \dd x \dd y \dd s,
\]
where $\tilde\Psi_n^n(s,x):=\tilde\Psi_n(n^2s,nx)$. To calculate the limit, we observe that 
\[
    \tilde\E[\cT_2]=\int_\eps^t\!\! \int_{\R^{2d}} w_n(x-y) q_{s}\ast\varphi(x) q_{s}\ast\varphi(y) \nabla \phi(x)^\ast Q_n (x-y) \nabla \phi(y) \dd x \dd y \dd s,
\]
where $w_n$ is the spatial covariance function of $\tilde\Psi_n^n$ (that is, of $\Psi(n^2t,nx)$). From \cite[Corollary 3.2, Proposition 2.3]{Dunlap-Gu} we can deduce that 
\[
    \tilde\E[\cT_2]\rightarrow \int_\eps^t\!\! \int_{\R^{d}} |q_{s}\ast\varphi(x)\, V_{\rm eff} \nabla \phi(x)|^2\dd x \dd s,
\]
where $V_{\rm eff}$ is defined in \eqref{eq:eff_var}. Observe that this shows that $\tilde\E[T_2]$ converges to the same limit. 

We want to show that
\[
    \tilde\E\bigg[\bigg|\cT_2- \int_\eps^t\!\! \int_{\R^{d}} |q_{s}\ast\varphi(x) \, V_{\rm eff} \nabla \phi(x)|^2\dd x \dd s\bigg|^2 \bigg] \rightarrow0.
\]
To do this, it suffices to show that
\begin{equation}\label{T-2-square}
    \tilde\E \big[\cT_2^2 \big]\to \bigg( \int_\eps^t\!\! \int_{\R^{d}} |q_{s}\ast\varphi(x)\, V_{\rm eff} \nabla \phi(x)|^2\dd x \dd s \bigg)^2\quad \mbox{as } n\to \infty.
\end{equation}
We can calculate the second moment as follows:
\begin{align*}
    \tilde\E \big[\cT_2^2 \big] &=\int_{[\eps,t]^2} \! \int_{\R^{4d}} w^{(4)}_n(r,s, x_{1:4}) (q_{r}\ast\varphi)^{\otimes 2}(x_{1:2}) (q_{s}\ast\varphi)^{\otimes 2}(x_{3:4}) \\ 
    & \times \nabla \phi(x_1)^\ast  Q_n (x_1-x_2) \nabla \phi(x_2) \nabla\phi(x_3)^\ast Q_n (x_3-x_4) \nabla \phi(x_4) \dd x_{1:4} \dd r \dd s,
\end{align*}
where 
  \[\aligned
  w^{(4)}_n(r,s, x_{1:4}) &= \E\big[\tilde\Psi_n^n(r,x_1) \tilde\Psi_n^n(r,x_2) \tilde\Psi_n^n(s,x_3) \tilde\Psi_n^n (s,x_4) \big] . \endaligned\]
Making change of variables $x_{1:4} \to (x_2+x_1/n, x_2, x_4+ x_3/n, x_4)$, we have
\begin{align}\label{eq:aux_term}
    \tilde\E \big[\cT_2^2 \big] &=\int_{[\eps,t]^2} \! \int_{\R^{4d}} \tilde  w^{(4)}_n(r,s, x_{1:4}) q_{r}\ast \varphi(x_{2}+x_1/n) q_{r}\ast\varphi(x_{2}) q_{s}\ast \varphi(x_{4}+x_3/n) q_{s}\ast\varphi(x_{4}) \nonumber\\
    &\times \nabla \phi(x_2+x_1/n)^\ast Q(x_1)  \nabla \phi(x_2)\nabla\phi(x_4+x_3/n)^\ast Q (x_3) \nabla \phi(x_4) \dd x_{1:4} \dd r \dd s,
\end{align}  
where now
  \[\tilde  w^{(4)}_n(r,s, x_{1:4}) = \E\big[\Psi(n^2 r, nx_2+x_1) \Psi(n^2 r, nx_2) \Psi(n^2 s, nx_4+x_3) \Psi(n^2 s, nx_4) \big]. \]
 By Lemma \ref{converg-correlation-funct}, stated and proved below, and the dominated convergence theorem, taking limit $n\to \infty$ yields
  \[\aligned
  \tilde\E \big[\cT_2^2 \big] &\to \int_{[\eps,t]^2} \! \int_{\R^{4d}} w(x_1) w(x_3) \big( q_{r}\ast \varphi(x_{2}) \big)^2 \big( q_{s}\ast \varphi(x_{4}) \big)^2 \\ 
  &\quad\times \nabla \phi(x_2)^\ast Q(x_1) \nabla \phi(x_2) \, \nabla\phi(x_4)^\ast Q (x_3) \nabla \phi(x_4) \dd x_{1:4} \dd r \dd s \\
  &= \bigg[\int_{[\eps, t]}\! \int_{\R^{2d}} \big( q_{r}\ast \varphi(x_{2}) \big)^2 \nabla \phi(x_2)^\ast (wQ)(x_1) \nabla \phi(x_2) \dd x_{1:2} \dd r \bigg]^2 \\
  &= \bigg[\int_{[\eps, t]} \! \int_{\R^{d}} \big( q_{r}\ast \varphi(x_{2}) \big)^2 \nabla \phi(x_2)^\ast V_{\rm eff}^2 \nabla \phi(x_2) \dd x_{2} \dd r \bigg]^2,
  \endaligned \]
which proves \eqref{T-2-square}. This concludes the proof.
\end{proof}

Observe that \textbf{Lemma \ref{lem:quad_var_lim}} implies that the limit of the martingale $\langle \tilde M_n(\cdot),\phi\rangle$ is a Gaussian process in $t$. Furthermore, from \textbf{Theorem $8.2$} in \cite{SPDEbook}, we can find a space time white noise $\xi:=(\xi_1,\dots,\xi_d)$, defined over the same probability space, such that
\[
    \langle\tilde M_n,\phi\rangle(t)=\int_0^t\langle\nabla\phi,V_{\rm eff}\, \bar\theta_s\, \xi(\dd s) \rangle,
\]
a.s. for all $\phi\in C^\infty_c(\R^d)$. Therefore, the limiting point $\tilde\cX$ satisfies the equation 
\[
    \cX_t(\phi)=(k+\nu)\int_0^t\cX_s(\Delta\phi)ds+\int_0^t\langle\nabla\phi, V_{\rm eff}\, \bar\theta_s\, \xi(\dd s) \rangle,
\]
a.s., for all $\phi\in C_c^\infty(\R^d)$. This characterizes the law of $\cX$ in $C^\gamma([0,T],H^{-\alpha-\epsilon}_{loc})$, which is given by the law of the weak solution to the additive heat equation \eqref{eq:limit_eq}. This proves \textbf{Theorem \ref{thm:Main_thm}}, under \textbf{Assumption \ref{Assump_compressible}}.\par

We conclude this section with the statement and proof of \textbf{Lemma \ref{converg-correlation-funct}}, referenced above.

\begin{lemma}\label{converg-correlation-funct}
The function $\tilde  w^{(4)}_n(r,s, x_{1:4})$ is uniformly bounded in $n\ge 1$, $r,s\in [\eps, t]$ and $x_{1:4}\in \R^{4d}$. Moreover, for any $\eps<r\neq s <t$, it holds
  \[\tilde  w^{(4)}_n(r,s, x_{1:4}) \to w(x_1) w(x_3) \quad \mbox{as } n\to \infty. \]
\end{lemma}

\begin{remark}\label{rmk:sol_neg_time}
    In the following proof, we need to work with certain solutions to \eqref{eq:transport_Ito} defined at negative times. For this, we introduce suitable filtrations. We extend the i.i.d. Brownian motions $B_k$ to i.i.d. two-sided Brownian motions, still denoted by $B_k$ (letting $B_k(t)=B^\prime_k(-t)$ for $t<0$, where $\{B^\prime_k\}_k$ is a family of i.i.d. Brownian motions independent on $\{B_k\}_k$). For $-\infty<s\le t<\infty$, we call $\cF_{s,t}$ the $\sigma$-algebra generated by $\{B_k(r)-B_k(s)\}_{k\in \mathbb N, r\in [s,t]}$ and the $\P$-null sets; we also call $\cG_t$ the sigma-algebra generated by $\{\cF_{s,t}\}_{s\le t}$.
    % Proceeding as in the proof of \cite[Proposition 2.5]{Bas2011}, we get that, for every $s$, $(\cF_{s,t})_{t\ge s}$ and $(G_t)_t$ are complete and right-continuous and $\{(B_k(s+r)-B_k(s))_{r\ge 0}\}_{k}$ is a family of i.i.d. Brownian motions both with respect to $\cF_{s,s+r}$ and to $(\G_{r+s})_r$.
    Clearly $\cG_s$ and $\cF_{s,t}$ are independent for every $s<t$.

    Following \cite{Dunlap-Gu}, we consider the solution $\theta^{[M]}$ to the transport-diffusion equation \eqref{eq:transport_Ito} starting at $t=-M$ with initial data $\theta^{[M]}(-M,\cdot) \equiv 1$, more precisely $\theta^{[M]}$ is defined by
    \[ \theta^{[M]} = \E\Big[\big(\det DX^{[M]}_t \big(X^{[M],-1}_t(x) \big) \big)^{-1}\, \big|\, \sigma(\cF_{-M,r},r\ge -M) \Big] \]
    where $X^{[M]}_t$ is the stochastic flow solving
    \[ dX^{[M]}_t(x) = \sum_k \sigma_k \big(X^{[M]}_t(x) \big) \circ \dd B_k(t) +\sqrt{2\kappa}\, dW(t),\quad X^{[M]}_{-M}(x)=x, \]
    with $W$ an independent $d$-dimensional Brownian motion. Then $\theta^{[M]}_t$ is $(\cF_{-M,t})$-measurable. Moreover, by \eqref{converg-second-moment} below, the stationary field $\Psi(t,\cdot)$ is $\cG_t$-measurable.
\end{remark}

\begin{proof}[Proof of Lemma \ref{converg-correlation-funct}]
Let $\theta^{[M]}(t,x)$ be the solution to \eqref{eq:transport_Ito} starting at $t=-M$ with initial data $\theta^{[M]}(-M,\cdot) \equiv 1$, as in Remark \ref{rmk:sol_neg_time}. Then it holds (cf. \cite[Corollary 3.2]{Dunlap-Gu})
  \begin{equation} \label{converg-second-moment}
  \lim_{M\to \infty} \E\big|\theta^{[M]}(t,x) - \Psi(t,x) \big|^2 =0. 
  \end{equation}
Using Proposition \ref{prop:corr_bound}, one can prove uniform estimates on higher order moments of $\theta^{[M]}(t,x)$. Indeed, for any $p\geq 2$ with $p\in \mathbb N$, defining the correlation function
  \[
  \cS_p(t+M,x_{1:p})= \E\big[ \theta^{[M]}(t, x_1)\ldots \theta^{[M]}(t, x_p) \big],  \quad t\ge -M,
  \]
then similarly to the proof of Lemma \ref{lemm:Lp_moment_bound}, using the fact $\theta^{[M]}(-M,\cdot) \equiv 1$ and Proposition \ref{prop:corr_bound}, we have
  \[\aligned
  |\cS_p(t+M,x_{1:p})| &= \bigg| \int_{\R^{dp}} G_p(t+M, x_{1:p}, y_{1:p}) \dd y_{1:p} \bigg| \\
  &\lesssim \int_{\R^{dp}} q_{c(t+M)}^{\otimes p}((x-y)_{1:p}) \dd y_{1:p} = 1.
  \endaligned \]
In particular, for any $p$ even and $x_1= \ldots = x_p= x\in \R^d$, we obtain the moment estimates
  \[ \E\big[ \big|\theta^{[M]}(t,x)\big|^p \big] \lesssim 1 \quad \mbox{uniformly in } t\ge 0, x\in \R^d. \]
Therefore, we deduce from \eqref{converg-second-moment} that
  \begin{equation} \label{convergence-high-order-moment}
  \lim_{M\to \infty} \E\big|\theta^{[M]}(t,x) - \Psi(t,x) \big|^p =0. 
  \end{equation}
Combining this limit with the above uniform bound, we can obtain the first assertion on $\tilde  w^{(4)}_n(r,s, x_{1:4})$.

Next, assuming $\eps<r<s <t$, we have, by temporal stationarity, 
  \[\aligned
  \tilde w^{(4)}_n(r,s, x_{1:4}) &= \E\big[\Psi(n^2(r-s), nx_2+ x_1) \Psi(n^2(r-s), nx_2)\Psi(0, nx_4+x_3) \Psi(0, nx_4) \big];
  \endaligned\]
moreover, by spatial stationarity, it holds
  \[ \aligned
  & \E\big[\Psi(n^2(r-s), nx_2+ x_1) \Psi(n^2(r-s), nx_2) \big] = w(x_1), \\
  & \E[\Psi(0, nx_4+x_3) \Psi(0, nx_4)] = w(x_3) .
  \endaligned \]
We have
  \[\aligned
  &\tilde w^{(4)}_n(r,s, x_{1:4}) - w(x_1) w(x_3) \\ 
  &= \E\big[ \big(\Psi(n^2(r-s), nx_2+ x_1) \Psi(n^2(r-s), nx_2) -w(x_1) \big) \\
  &\qquad \times \big(\Psi(0, nx_4+x_3) \Psi(0, nx_4) -w(x_3) \big) \big]\\
  &= \E\big[ \big(\Psi(n^2(r-s), nx_2+ x_1) \Psi(n^2(r-s), nx_2) - w(x_1) \big) \\
  &\qquad \times \big(\Psi(0, nx_4+x_3) \Psi(0, nx_4) - \theta^{[n^2(r-s)]}(0, nx_4+ x_3) \theta^{[n^2(r-s)]}(0, nx_4) \big) \big] \\ 
  &+ \E\big[ \big(\Psi(n^2(r-s), nx_2+ x_1) \Psi(n^2(r-s), nx_2) - w(x_1) \big) \\
  &\qquad \times \big(\theta^{[n^2(r-s)]}(0, nx_4+ x_3) \theta^{[n^2(r-s)]}(0, nx_4) -w(x_3) \big) \big] .
  \endaligned \]
Note that by construction, the stationary field $\Psi(n^2(r-s), x)$ depends on information of the driving field $V(t,\cdot)$ for $t<n^2(r-s) <0$ (precisely, it is $\cG_{n^2(r-s)}$-measurable), while $\theta^{[n^2(r-s)]}(0, x)$ depends on the information of $V(t,x)$  for $t\geq n^2(r-s)$ (precisely, it is $\cF_{n^2(r-s),0}$-measurable); thus, they are independent. As a result, the second expectation is equal to
  \[\aligned 
  &\E\big[ \Psi(n^2(r-s), nx_2+ x_1) \Psi(n^2(r-s), nx_2) - w(x_1) \big] \\
  &\quad \times \E\big[ \theta^{[n^2(r-s)]}(0, nx_4+ x_3) \theta^{[n^2(r-s)]}(0, nx_4) -w(x_3) \big] =0.
  \endaligned \]
Concerning the first expectation, we have, by spatial stationarity,
  \[\aligned
  & \E\big| \Psi(0, nx_4+x_3) - \theta^{[n^2(r-s)]}(0, nx_4+ x_3) \big|^4 \\ 
  &= \E\big| \Psi(0, 0) - \theta^{[n^2(r-s)]}(0, 0) \big|^4 \to 0 \quad \mbox{as } n\to \infty  , 
  \endaligned \]
where the last step is due to \eqref{convergence-high-order-moment}; similarly, as $n\to \infty$,
  \[\aligned
  & \E\big| \Psi(0, nx_4) - \theta^{[n^2(r-s)]}(0, nx_4) \big|^4  = \E\big| \Psi(0, 0) - \theta^{[n^2(r-s)]}(0, 0) \big|^4 \to 0 .
  \endaligned \]  
Therefore, by H\"older's inequality, one can show that the first expectation on the right-hand side of $\tilde w^{(4)}_n(r,s, x_{1:4}) - w(x_1) w(x_3)$ also vanishes. 

Finally, the case $\varepsilon <s<r<t$ can be treated in the same way; as a consequence, we deduce that, for all $r\neq s$, 
  \[\tilde w^{(4)}_n(r,s, x_{1:4}) - w(x_1) w(x_3) \to 0 \]
as $n\to \infty$. 
\end{proof}

\subsection{Proof of \textbf{Theorem \ref{thm:Main_thm}}: The incompressible case}\label{Sec:Proof_inc_compr} 
Here we prove \textbf{Theorem \ref{thm:Main_thm}}, under \textbf{Assumption \ref{assump-covariance}}; in this case, $V_{\rm eff}^2= g(0)\Pi$, where the function $g$ is given in \eqref{covar-Fourier} and $\Pi$ is the Helmholtz-Leray projection.  

We follow the same steps as before. As mentioned in \textbf{Remark \ref{rem:quant_est_divfree}}, \textbf{Proposition \ref{prop:quant_est}} holds in this case as well. As such, the first steps in the previous section can be followed verbatim. In particular, \textbf{Proposition \ref{prop-tightness-X-n}} and \textbf{Proposition \ref{prop-tightness-M-n}} hold under \textbf{Assumption \ref{assump-covariance}} as well. As such, it remains to  control the martingales
\[
    \tilde M_n(t) = n^{d/2} \sum_k \int_0^t \nabla\cdot(\sigma^n_k \tilde \theta^n_s) \dd\tilde B^n_k(s)=n^{d/2} \sum_k \int_0^t (\sigma^n_k  \cdot\nabla\tilde \theta^n_s) \dd\tilde B^n_k(s),
\]
where we have already used Skorohod's representation theorem, and the fact that \(\sigma^n_k(x)= \sigma_k(nx)\) is divergence-free. We still have \eqref{eq-quadratic-var}, which can be written in a more compact form:
\begin{equation}\label{eq:aux_eq}
    [\<\tilde M_n, \phi\> ]_t= \int_0^t \big\<\tilde \theta^n_s \nabla \phi, Q_n\ast (\tilde \theta^n_s \nabla \phi) \big\> \dd s,
\end{equation}
where $Q_n(x)= n^d Q(nx)$, $ x\in \R^d$. We wish to show that 
\begin{equation}\label{eq:wish_conc}
    [\<\tilde M_n, \phi\> ]_t\rightarrow\int_0^t\!\! \int_{\R^{d}} (\theta_s \nabla \phi)(x)^\ast V_{\rm eff}^2 (\theta_s \nabla \phi)(x) \dd x \dd s,
\end{equation}
in probability, where now, \(V_{\rm eff}^2= g(0)\Pi \). 

We split the right-hand side of \eqref{eq:aux_eq} as follows:
\begin{align}\label{eq:splitting_quadratic_var}
  \aligned 
  [\<\tilde M_n, \phi\> ]_t 
  &= \int_0^t \big\<(\tilde \theta^n_s -\bar \theta_s) \nabla \phi, Q_n\ast (\tilde \theta^n_s \nabla \phi) \big\> \dd s 
  + \int_0^t \big\< \bar \theta_s \nabla \phi, Q_n\ast (\tilde \theta^n_s \nabla \phi) \big\> \dd s \\
  &= \int_0^t \big\<(\tilde \theta^n_s -\bar \theta_s) \nabla \phi, Q_n\ast ((\tilde \theta^n_s -\bar \theta_s) \nabla \phi) \big\> \dd s \\ 
  &\quad + \int_0^t \big\<(\tilde \theta^n_s -\bar \theta_s) \nabla \phi, Q_n\ast (\bar \theta_s \nabla \phi) \big\> \dd s \\
  &\quad  + \int_0^t \big\< \bar \theta_s \nabla \phi, Q_n\ast ((\tilde \theta^n_s- \bar \theta_s) \nabla \phi) \big\> \dd s \\
  &\quad + \int_0^t \big\< \bar \theta_s \nabla \phi, Q_n\ast ( \bar \theta_s \nabla \phi) \big\> \dd s .
  \endaligned
  \end{align}
We denote the last four terms by $J_i,\, i=1,2,3,4$. First, one has
  \[ J_4= \int_0^t \big\< \bar\theta_s \nabla \phi, Q_n\ast ( \bar\theta_s \nabla \phi) - V_{\rm eff}^2( \bar\theta_s \nabla \phi) \big\> \dd s + \int_0^t \big\< \bar \theta_s \nabla \phi, V_{\rm eff}^2( \bar\theta_s \nabla \phi) \big\> \dd s. \]
The first part of $J_4$ can be estimated as
  \[\aligned 
  &\int_0^t \big\| \bar\theta_s \nabla \phi \big\|_2 \big\| Q_n\ast ( \bar\theta_s \nabla \phi) - V_{\rm eff}^2( \bar\theta_s \nabla \phi) \big\|_2 \dd s \\ &\le \|\varphi \|_2 \|\nabla\phi \|_\infty \int_0^t \big\| Q_n\ast ( \bar\theta_s \nabla \phi) - V_{\rm eff}^2( \bar\theta_s \nabla \phi) \big\|_2 \dd s .
  \endaligned \]
%\todo{SK: It seems that we only need the last term to go to 0. Probably, we can get away with this by replacing \(V_{\rm eff}\) by \( \widehat{Q}(0)\), without requiring \(Q\) to be integrable, since \(\widehat{Q}\in L^1\cap L^\infty\). Notice that when \(Q\) is integrable this does not change anything. Dejun: following Sotirios' idea, I have made corrections to the proof.} 
By the Plancherel identity, 
  \[\aligned
  \big\| Q_n\ast ( \bar\theta_s \nabla \phi) - V_{\rm eff}^2( \bar\theta_s \nabla \phi) \big\|_2^2 &= \int_{\R^d} \big|\widehat{Q}_n \F(\bar\theta_s \nabla \phi) - g(0) P_\xi \F(\bar\theta_s \nabla \phi)\big|^2\dd\xi\\
  &= \int_{\R^d} \big|\big(g(n^{-1}\xi) -g(0) \big) P_\xi \F(\bar\theta_s \nabla \phi)\big|^2\dd\xi ,
  \endaligned \]
where we have used $\widehat{Q}_n(\xi)= \widehat{Q}(n^{-1}\xi)= g(n^{-1}\xi) P_\xi$ by \eqref{covar-Fourier}. Recall that $g\in C_b(\R^d)$; for any $s\in [0,t]$ and $\xi\ne 0$, it is clear that the integrand vanishes as $n\to\infty$, thus by the dominated convergence theorem, one has
  \[\lim_{n\to\infty} \big\| Q_n\ast ( \bar\theta_s \nabla \phi) - V_{\rm eff}^2( \bar\theta_s \nabla \phi) \big\|_2=0. \]
Therefore, the first term of $J_4$ vanishes as $n\to \infty$.

Next, for $J_2$, in the same way we have
  \[\aligned J_2 &= \int_0^t \big\<(\tilde \theta^n_s -\bar\theta_s) \nabla \phi, Q_n\ast (\bar\theta_s \nabla \phi)- V_{\rm eff}^2( \bar\theta_s \nabla \phi) \big\> \dd s \\ 
  &\quad + \int_0^t \big\<(\tilde \theta^n_s -\bar\theta_s) \nabla \phi, V_{\rm eff}^2( \bar\theta_s \nabla \phi) \big\> \dd s \endaligned\]
which will be denoted as $J_{2,1}$ and $J_{2,2}$. We can regard $ \nabla\phi\cdot V_{\rm eff}^2( \bar\theta_s \nabla \phi)$ as a test function ($\bar\theta_s$ is smooth since it solves the heat equation \eqref{eq:heat}), thus by item (b) above, it is clear that $J_{2,2}$ tends to 0 as $n\to \infty$. Concerning $J_{2,1}$, we have
  \[\aligned 
  |J_{2,1}| &\le \int_0^t \big\|(\tilde \theta^n_s - \bar \theta_s) \nabla \phi \big\|_2  \big\| Q_n\ast (\bar\theta_s \nabla \phi)- V_{\rm eff}^2( \bar\theta_s \nabla \phi) \big\|_2  \dd s \\
  &\le \|\nabla\phi \|_\infty \int_0^t \big( \|\tilde \theta^n_s \|_2  +\| \bar\theta_s\|_2 \big)  \big\| Q_n\ast (\bar\theta_s \nabla \phi)- V_{\rm eff}^2( \bar\theta_s \nabla \phi) \big\|_2  \dd s \\
  &\le 2 \|\nabla\phi \|_\infty \| \varphi \|_2 \int_0^t \big\| Q_n\ast (\bar\theta_s \nabla \phi)- V_{\rm eff}^2( \bar\theta_s \nabla \phi) \big\|_2  \dd s
  \endaligned \]
which, similarly to the treatment of the first term in $J_4$, also vanishes as $n\to \infty$. In the same way, we can show that $J_3 \to 0$. 

Finally, to handle the term $J_1$, we first prove the following lemma.

\begin{lemma}\label{lem-L-2-convergence}
Let $\kappa>0$, $\theta^n$ be solution to 
  \[\dd \theta^n+ \circ \dd W^n \cdot \nabla \theta^n = \kappa \Delta \theta^n \dd t, \quad \theta^n_0= \theta_0\]
and $\bar \theta$ the solution to 
  \[\partial_t \bar \theta = (\kappa+\nu) \Delta \bar \theta, \quad \bar \theta_0= \theta_0.\]
Then one has
  \[ \E \int_0^T \|\theta^n_t -\bar \theta_t \|_2^2 \dd t \lesssim_T \|\theta_0 \|_2^{2} \big\| \widehat{Q} \big\|_\infty^{1-\delta} n^{-d(1-\delta)} , \]
where $\delta= (d+2\eps)/(2+d+2\eps)$.
\end{lemma}

\begin{proof}
By \textbf{Proposition \ref{prop:quant_est}} (and \textbf{Remark \ref{rem:quant_est_divfree}}), for some $\eps\in (0,1)$, we have
  \[ \E\|\theta^n_t -\bar \theta_t \|_{\dot H^{-d/2-\eps}}^2 \lesssim_T \|\theta_0\|_2^2 \big\| \widehat{Q} \big\|_\infty n^{-d} \]
for all $t\in [0,T]$. Note that $\theta^n$ and $\bar \theta$ satisfies the following energy estimates:
  \[\aligned
  \P \mbox{-a.s.}, \quad \|\theta^n_t \|_2^2 + 2\kappa \int_0^t \|\nabla \theta^n_s\|_2^2 \dd s & \le 2\|\theta_0 \|_2^2, \\
  \|\bar \theta_t \|_2^2 + 2(\kappa +\nu) \int_0^t \|\nabla \bar \theta_s\|_2^2 \dd s &= \|\theta_0 \|_2^2.
  \endaligned \]
The first estimate is a consequence of \eqref{eq:energy_est}, since \(\theta^n\) satisfies  \eqref{eq:transport_Ito} with a rescaled noise term. We can get the second energy estimate by integrating by parts. Combining these estimates yields
  \[ \P \mbox{-a.s.}, \quad \int_0^T \|\theta^n_t -\bar \theta_t\|_{\dot H^1}^2 \dd t \lesssim_{\kappa,T} \|\theta_0 \|_2^2. \]
By interpolation, $\|\theta^n_t -\bar \theta_t\|_2 \le \|\theta^n_t -\bar \theta_t\|_{\dot H^1}^\delta \|\theta^n_t -\bar \theta_t \|_{\dot H^{-d/2-\eps}}^{1-\delta} $, where $\delta= (d+2\eps)/(2+d+2\eps)$. As a result, by Cauchy's inequality,
  \[\aligned
  \E \int_0^T \|\theta^n_t -\bar \theta_t \|_2^2 \dd t & \le \E \int_0^T \|\theta^n_t -\bar \theta_t\|_{\dot H^1}^{2\delta} \|\theta^n_t -\bar \theta_t \|_{\dot H^{-d/2-\eps}}^{2(1-\delta)} \dd t \\
  &\le \bigg[\E \int_0^T \|\theta^n_t -\bar \theta_t\|_{\dot H^1}^{2} \dd t \bigg]^\delta \bigg[\E \int_0^T \|\theta^n_t -\bar \theta_t\|_{\dot H^{-d/2-\eps}}^{2} \dd t \bigg]^{1-\delta} \\
  &\lesssim_T \|\theta_0 \|_2^{2} \big\| \widehat{Q} \big\|_\infty^{1-\delta} n^{-d(1-\delta)} .
  \endaligned\]
\end{proof}

Now we can estimate $J_1$ as follows:
  \[ \aligned
  \tilde \E|J_1| &\le \tilde \E \int_0^T \big\|(\tilde \theta^n_s - \bar \theta_s)\nabla\phi \big\|_2 \big\|Q_n \ast ((\tilde \theta^n_s - \bar \theta_s)\nabla\phi )\big\|_2 \dd s; 
  \endaligned \]
by Plancherel's identity and \eqref{covar-Fourier},
  \[\aligned
  \big\|Q_n \ast ((\tilde \theta^n_s - \bar \theta_s)\nabla\phi )\big\|_2^2 &= \int_{\R^d} \big|\widehat{Q}_n(\xi) \F((\tilde \theta^n_s - \bar \theta_s)\nabla\phi )(\xi)\big|^2 \dd\xi \\
  &\le \|g\|_\infty \int_{\R^d} \big| \F((\tilde \theta^n_s - \bar \theta_s)\nabla\phi )(\xi)\big|^2 \dd\xi \\
  &= \|g\|_\infty \big\|(\tilde \theta^n_s - \bar \theta_s)\nabla\phi \big\|_2^2 .
  \endaligned\]
Plugging this estimate in the above inequality yields
  \[ \aligned
  \tilde \E|J_1| &\le \|g\|_\infty \tilde \E\! \int_0^T \! \big\|(\tilde \theta^n_s - \bar \theta_s)\nabla\phi \big\|_2^2 \dd s \le \|g\|_\infty \|\nabla\phi \|_\infty \tilde\E\! \int_0^T \! \big\|\tilde \theta^n_s - \bar \theta_s \big\|_2^2 \dd s .
  \endaligned \]
By Lemma \ref{lem-L-2-convergence} and the fact that $\tilde \theta^n$ has the same law as $\theta^n$, we conclude that the last quantity vanishes as $n\to \infty$.\par

This shows \eqref{eq:wish_conc}. We can continue in the same way as in the previous section (after the proof of \textbf{Lemma \ref{lem:quad_var_lim}}) and conclude the proof of \textbf{Theorem \ref{thm:Main_thm}} under \textbf{Assumption \ref{assump-covariance}} as well.

\bigskip

\noindent \textbf{Acknowledgements.} The first author acknowledges the partial support of the project \emph{Noise in fluid dynamics and related models}, funded by the Italian Ministry of University and Research through the project PRIN call 2022 - grant 20222YRYSP. The second author is grateful to the financial supports of the National Key R\&D Program of China (No. 2024YFA1012300) and the National Natural Science Foundation of China (Nos. 12090010, 12090014).  The third author acknowledges the partial support of the project \emph{Convergence and Stability of Reaction and Interaction Network Dynamics}, funded by the Italian Ministry of University and Research through the project PRIN call 2022 - grant 2022XRWY7W.

\end{document}